\documentclass[11pt]{amsart}
\usepackage{amssymb, amscd, amsmath, amsthm, epsf, epsfig, latexsym,psfrag,color} 
\usepackage{hyperref}
\usepackage{float}
\usepackage{psfrag}
\usepackage{comment}

\usepackage{pb-diagram}
\usepackage{graphicx}

\newcommand{\tr}{\mbox{tr}}

\newcommand{\TryPackage}[3]{\IfFileExists{#1.sty}{\usepackage{#1}#2}{#3}
}
\TryPackage{mathrsfs}{\renewcommand{\mathcal}{\mathscr}}{%
        \TryPackage{eucal}{}{}}

\newcommand{\al}{\alpha}

\newcommand{\bbi}{{{\bf i}}}
\newcommand{\bbj}{{{\bf j}}}
\newcommand{\bbk}{{{\bf k}}}

\newcommand{\ZZ}{{\mathbb Z}}
\newcommand{\RR}{{\mathbb R}}
\newcommand{\CC}{{\mathbb C}}

\newcommand{\QQ}{{\mathbb Q}}

\newcommand{\Hom}{\operatorname{Hom}}

\newcommand{\nat}{\natural}

\newcommand{\Real}{\operatorname{Re}}

\graphicspath{ {figures/} }

\theoremstyle{definition}

\newtheorem{df}{Definition}[section]

\theoremstyle{plain}

\newtheorem{thm}[df]{Theorem}

\newtheorem{prop}[df]{Proposition}

\date{May 2, 2013}

\thanks{The first author gratefully acknowledges JSPS Grant-in-Aid for
Scientific Research (C) 23540113.  The second author gratefully acknowledges support from 
NSF  grant DMS-1007196.}

\address{Department of Mathematics, Indiana University \newline
\hspace*{.375in}  Bloomington, IN 47405} 
\email{\rm{yfukumot@fc.ritsumei.ac.jp}}
\email{\rm{pkirk@indiana.edu}}
\email{\rm{jpinzon@indiana.edu}}


\begin{document}

\title {Traceless $SU(2)$ representations of 2-stranded tangles}

\author{Yoshihiro Fukumoto}
\author{Paul Kirk}
\author{Juanita Pinz\'on-Caicedo}
\begin{abstract} Given a 2-stranded tangle in a $\ZZ/2$ homology ball, $T\subset Y$,  we  investigate the character variety $R(Y,T)$ of conjugacy classes of traceless $SU(2)$ representations of $\pi_1(Y\setminus T)$.    In particular we  completely determine   the subspace of binary dihedral representations, and identify all of  $R(Y,T)$ for many tangles naturally  associated to knots in $S^3$.   Moreover, we determine the image of the restriction map from $R(T,Y)$ to  the traceless $SU(2)$ character variety of the 4-punctured 2-sphere (the {\it pillowcase}).  We give examples to show this image can be non-linear in general, and show it is linear for tangles associated to pretzel knots.    \end{abstract}

 \maketitle

 \section{Introduction}

 In \cite{HHK}, a study of the moduli spaces which arise in the construction of Kronheimer-Mrowka's reduced singular instanton knot homology (\cite{KM1,KM-khovanov,KM-filtrations}) naturally led to an exploration of a lagrangian intersection diagram associated to a decomposition of a pair $(Y,K)$, where $K$ is a knot in a closed 3-manifold $Y$, and $Y$ contains a {\em Conway sphere}, that is, a separating  2-sphere $S\subset Y$ which intersects $K$ transversally in four points.  To the  decomposition

\begin{equation}
\label{decom}
(Y,K)=(Y_0,T_0)\cup _{(S,\{a,b,c,d\})}(Y_1,T_1)
\end{equation}
one associates the diagram 
 \begin{equation}\label{SVKD}
\begin{diagram}\node[2]{R(S,\{a,b,c,d\})}\\
\node{R(Y_0,T_0)}\arrow{ne}\node[2]{R^\nat_\pi(Y_1,T_1)}\arrow{nw}\\
\node[2]{R_\pi^\nat(Y,K)}\arrow{nw}\arrow{ne}
\end{diagram}
\end{equation}
 where $R_\pi(M,L)$ denotes certain $\pi$-perturbed moduli spaces of singular flat $SU(2)$ connections which send meridians of $L$ to traceless matrices.  The points of the moduli space $R^\nat_\pi(Y,K)$ form the generators of the reduced instanton chain complex $CI^\nat(Y,K)$.
 See Figure \ref{fig45a} for the example corresponding to a tangle decomposition of the (4,5) torus knot.
 \bigskip

 In \cite{HHK}, 
 \begin{itemize}
\item $K\subset Y$ is usually taken to be a knot in $S^3$ and $T_1$ is the  trivial 2-stranded tangle in the 3-ball $Y_1=B^3$, so that $Y_0$ is the complementary 3-ball and $T_0$ is (in general) a non-trivial 2-stranded tangle. 
\item The space $R(S,\{a,b,c,d\})$ is identified with the {\em pillowcase}, a topological 2-sphere with four singular points arising as the quotient of the hyperelliptic involution on the torus.
\item An appropriate holonomy perturbation $\pi$ is constructed which makes $R^\nat_\pi(Y_1,T_1)$ regular (in fact a smooth circle) and the image of the restriction map $R^\nat_\pi(Y_1,T_1)\to  R(S,\{a,b,c,d\})$ is identified (see Figure \ref{pert}).
\end{itemize}

Completing the analysis of diagram \eqref{SVKD} is reduced to the study of the restriction map  $R(Y_0,T_0)\to  R(S,\{a,b,c,d\})$. This is a problem purely about the geometry of traceless representation varieties, as all the difficulties concerning singular connections and perturbations are contained in  the right side of the decomposition of diagram (\ref{SVKD}). Studying the space $R(Y_0,T_0)$ and the restriction map to the pillowcase leads  to some interesting geometry. This article, then, explores the problem of identifying the map on the upper left side of this diagram, for various knots $K$ in the 3-sphere.

 The analysis was carried out  completely in  \cite{HHK} for the case when $K$ is a 2-bridge knot and $Y_1$ is the 3-ball containing the two bridges. The situation for $K$ a $(p,q)$ torus knot was initiated in that article, and  for $p,q>2$ is quite complicated. In fact only the case of the $(3,4)$ torus knot was fully described. 
 
  Based on those few examples, it was speculated in \cite{HHK} that 
the space $R(Y_0,T_0)$ always contains a single arc of binary dihedral representations when $(Y_0,T_0)$ is a tangle obtained by removing a trivial tangle from a knot in $S^3$, and that the image of the restriction map $R(Y_0,T_0)\to R(S^2,\{a,b,c,d\})$ is linear (in the sense that the preimage in the universal orbifold cover $\RR^2\to R(S^2,\{a,b,c,d\})$  is contained in the union of straight lines).

 In this article we show that both of these assertions are {\em  false}. Indeed, we completely characterize the subspace of binary dihedral representations for any 2-stranded tangle,  and show that it can have arbitrarily many components.  We give many examples of torus knots for which the image in the pillowcase is non-linear.  We also identify another family of tangles, naturally associated to pretzel knots, for which the image {\em is} linear on both the binary dihedral and non-binary dihedral components of $R(Y_0,T_0)$.
 
   \medskip
 
  The restriction to traceless representations is motivated by Kronheimer-Mrowka's profound work; one can make more general definitions by restricting the trace of the meridians in other ways. In this article we  implicitly exploit a useful relationship between traceless $SU(2)$ representations of a knot or tangle complement and  $SO(3)$ representations  of the corresponding  2-fold branched cover.

A closely related and slightly simpler picture arises in the context of
Lin's theorem \cite{Lin} and its generalizations \cite{HK,Herald}. In that situation,  a signed count of irreducible traceless $SU(2)$ representations of $\pi_1(Y\setminus K)$ for $K$ a knot in a homology 3-sphere $Y$ is identified with $-4\lambda(Y)-\frac{1}{2}\sigma(K)$, where $\lambda(Y)$ denotes Casson's invariant and $\sigma(K)$ the signature of  $K$. The irreducible traceless representations correspond (when $R(Y_0,T_0)$ is suitably regular; see the discussion of \cite[Section7.2]{HHK}) to the intersection points in the pillowcase of  $R(Y_0,T_0)$ with  the interior of the arc $\{\gamma=\theta\}$.  

Also related is the work of Jacobsson-Rubinsztein \cite{JR}; their approach extends \cite{Lin,HK} by considering a knot as the closure of a braid, and identifying $R(Y\setminus K)$ as the fixed point set of the braid group element, viewed as a diffeomorphism of a $2n$-punctured 2-sphere, acting on the symplectic variety $R(S^2\setminus\{a_1,\cdots,a_{2n}\})$. In that work they aim to apply symplectic field theory methods to recover information about the Khovanov homology of $K$, extending the observations of \cite{Lew}.

 One can summarize the difference in approaches of the present article and \cite{JR,Lin,HK} 
as follows: in the present article (and \cite{HHK}), $(Y,K)$ is decomposed along a 4-punctured 2-sphere into a trivial 2-stranded tangle and a general tangle in a ($\ZZ/2$-homology) ball, whereas in  \cite{JR,Lin,HK} it is decomposed into  trivial $n$-stranded tangles in 3-balls, a trivial and a non-trivial braid, and the braid group element determines how these are glued. This distinction is analogous to the description of a 3-manifold in terms of surgery on a knot or  by a Heegaard splitting. One feature of the present approach is that it gives low dimensional pictures: immersed 1-complexes $R(Y_0,T_0)$ in the 2-dimensional pillowcase.

\medskip

  Independently of their relationship to instanton homology or lagrangian intersection, the geometry of the representation spaces $R(Y_0,T_0)$ and their restriction to the pillowcase  display a variety of interesting  features. 
  We are guided in our investigation by the rich parallel theory concerning $SU(2)$ (and $SL(2,\CC)$) character varieties of knot complements and their restriction to the  peripheral torus, as studied in many articles, including \cite{Klassen, Klassen-Whitehead, KK, Burde, Herald-thesis, ccgls}.

 \bigskip
 
 The contents of this article are as follows. In Section \ref{sec1}   we remind the reader of the results of \cite{HHK} and in Proposition \ref{symmetry}  find a $\left(\ZZ/2\right)^2$ action on $R(Y_0,T_0)$ and a $\ZZ/2$ action on $R(S^2,\{a,b,c,d\})$ and show that the restriction map is suitably equivariant.

  In Theorem \ref{bd} we completely characterize  the subspace $R^{tbd}(Y_0,T_0)\subset R(Y_0,T_0)$ of  traceless binary dihedral representations and show that the restriction to the pillowcase is linear on this subspace. In the following statement, $\tilde Y_0$ denotes the 2-fold   cover of $Y_0$ branched over $T_0$, and $A_-\subset H_1(\tilde Y_0)$ denotes the intersection of the torsion subgroup with the $-1$ eigenspace of the covering involution.  The pillowcase is  identified with the quotient of the affine action $(\gamma,\theta)\mapsto \pm (\gamma+2\pi m,\theta+2\pi n)$ of $\ZZ^2\ltimes \ZZ/2$ on $\RR^2$ with fundamental domain $[0,\pi]\times [0,2\pi]$.
 
 \medskip
 
 \noindent{\bf Theorem 3.2.}  {\em  The subspace   $R^{tbd}(Y_0,T_0)\subset R(Y_0,T_0)$ consisting of $SU(2)$ conjugacy classes of traceless binary dihedral representations is homeomorphic to the disjoint union of one arc and $\frac{n-1}2$ circles, where $n$ is the order of $A_-$.  
 
The arc maps to  the pillowcase  by the  map
\begin{equation*}
\label{bdpip}
t\in [0,\pi]\mapsto (\gamma,\theta)=(h(ba)t, (h(ba)-h(bc^{-1}))t)
\end{equation*}
and each circle maps to the pillowcase by the    map
\begin{equation*}
\label{bdpip2}
e^{t\bbi}\in S^1\mapsto(\gamma,\theta)=(h(ba)t, (h(ba)-h(bc^{-1}))t)-(  \ell_1,\ell_2),
\end{equation*}
where $h\in \Hom(\pi_1(\tilde Y_0),\ZZ)=H^1(\tilde Y_0)\cong \ZZ$ denotes the generator and $(\ell_1,\ell_2)$ depends on the circle.}

 \medskip

 In Section 4 we focus on certain natural tangles associated to  torus knots.
  We expand and clarify the results obtained in \cite{HHK}, and in particular show with explicit examples that the image in the pillowcase can be non-linear. The reader should glance at Figure \ref{fig45}, which illustrates $R(Y_0,T_0)$ and its image in the pillowcase for a tangle associated to the $(4,5)$ torus knot. This is an interesting knot from the perspective of instanton homology, Heegaard-Floer homology, and Khovanov homology, 
and its image as an immersed lagrangian subvariety of the pillowcase is not contained in a union of straight line segments.  Our calculations provide a rich collection of lagrangian intersection pictures which correspond to the symplectic analogue of Kronheimer-Mrowka's instanton homology of knots. We hope that future work, guided by these examples, will motivate the investigation of an Atiyah-Floer conjecture for instanton knot homology.  In this direction  we observe that in the  cases when the closure of the non-binary dihedral representations misses the binary dihedral representations in $R(Y_0,T_0)$ (including, for example, the $(3,7), (5,7), (5,12),$ and $ (5,17)$ torus knots) then all differentials in the reduced instanton chain complex vanish. We provide examples that show that $R(T_0,Y_0)$ and its image in the pillowcase can be very complicated.

  In Section 5 we analyze a family of tangles naturally associated to pretzel knots. We show how to  associate two slopes to each such tangle  such that that the non-binary dihedral part of  $R(Y_0,T_0)$  is mapped to straight lines of one slope, and the binary dihedral part maps to straight line segments of  the other slope.  
  We finish with a few comments  concerning extensions to $SL(2,\CC)$ character varieties and some speculation about differentials in the instanton complex.

  The authors thank Chris Herald and Matt Hedden for helpful conversations.  The first author was a visiting scholar of Department of Mathematics at
Indiana University Bloomington while this work was carried out. He wishes to thank the department for
its hospitality.
 
 \section{Representations of tangles}\label{sec1} Identify $SU(2)$ with the unit quaternions.  We recall the set up from \cite{HHK}.  Let $Y_0$ be a homology 3-ball
 $T_0\subset Y_0$ a 2-stranded tangle,  such that the first component of $T_0$ meets the boundary 2-sphere in points labelled $a,c$ and the second component meets in points labelled $b,d$, as illustrated in Figures \ref{torusknot} and \ref{pqrpretzel}. We abuse notation and denote the four meridian loops in $\pi_1(S^2\setminus \{a,b,c,d\})$ (with respect to a base point on the 2-sphere)    and their images in $\pi_1(Y_0\setminus T_0)$ also by $a,b,c,d$.  The base point is chosen so that   $$\pi_1(S^2\setminus \{a,b,c,d\})=\langle a,b,c,d~|~ba=cd\rangle.$$
 
In what follows, $(Y_1,T_1)$ will always denote the trivial 2-stranded tangle in the 3-ball $Y_1$, as illustrated in Figure \ref{trivial}; hence $\pi_1(Y_1\setminus T_1)$ is free on $a=d$ and $b=c$.

  \begin{figure}  \centering
\def\svgwidth{2in}
 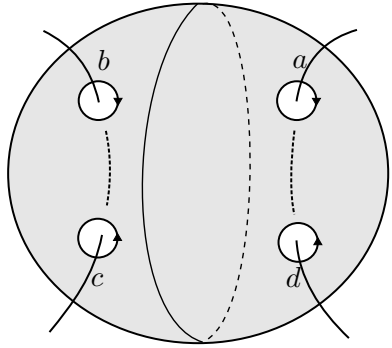
 \caption{The trivial tangle in the 3-ball, $(Y_1,T_1)$ \label{trivial}}
\end{figure}

\medskip

 We call an $SU(2)$  representation $\rho:\pi_1(Y_0\setminus T_0)\to SU(2)$ {\em traceless} if 
 $\rho$ send each meridian of $T_0$ to  the set (in fact conjugacy class) of purely imaginary quaternions (the trace on $SU(2)$ equals twice the real part of a quaternion).  
The space of conjugacy classes of traceless $SU(2)$ representations of $\pi_1(Y_0\setminus T_0)$ is denoted by $R(Y_0,T_0)$.   Similarly the space of conjugacy classes of traceless $SU(2)$ representations of $\pi_1(S^2\setminus \{a,b,c,d\})$ is denoted by $R(S^2, \{a,b,c,d\})$ and the
space of conjugacy classes of traceless $SU(2)$ representations of $\pi_1(Y_1\setminus T_1)$ is denoted by $R(Y_1,T_1)$.

The space $R(S^2, \{a,b,c,d\})$ is identified with the {\em pillowcase}; see \cite[Lemma 2.1]{Lin}  and \cite[Lemma 4.1]{HK}.    In \cite[Proposition 3.1]{HHK}  the map $\RR^2\to R(S^2, \{a,b,c,d\})$ taking the pair $(\gamma,\theta)$ to the conjugacy class of 
 \begin{equation}
\label{eqpil}a\mapsto\bbi, ~b\mapsto e^{\gamma\bbk}\bbi,~c\mapsto e^{\theta\bbk}\bbi~b\mapsto e^{(\theta-\gamma)\bbk}\bbi
\end{equation}
is shown to be a surjective quotient map. A fundamental domain is $[0,\pi]\times[0,2\pi]$ and the restriction 
$[0,\pi]\times[0,2\pi]\to R(S^2, \{a,b,c,d\})$ is given by the  identifications
$$(\gamma,0)\sim(\gamma,2\pi),~ (0,\theta)\sim(0,2\pi-\theta),~ (\pi,\theta)\sim(\pi,2\pi-\theta).$$

\medskip

The pillowcase is homeomorphic to a 2-sphere with four distinguished points corresponding to the abelian representations. In all the figures in this article we illustrate the pillowcase by  drawing the fundamental domain $[0,\pi]\times[0,2\pi]$; the image of the restriction map $R(Y_0,T_0)\to R(S^2, \{a,b,c,d\})$ is then illustrated as an immersed 1-complex in this rectangle.

\medskip

There are well defined restriction maps $R(Y_0,T_0)\to R(S^2, \{a,b,c,d\})$ and $R^\nat(Y_1,T_1)\to R(S^2, \{a,b,c,d\})$.  The present article concerns itself with the problem of identifying the image $R(Y_0,T_0)\to R(S^2, \{a,b,c,d\})$.  We mention, however, that $R^\nat(Y_1,T_1)$  is homeomorphic to a 2-sphere (\cite[Proposition 6.2]{HHK}) and  the restriction $R^\nat(Y_1,T_1)\to R(S^2, \{a,b,c,d\})$ is a quotient  map    with image the straight arc $\gamma=\theta$. The main result of \cite{HHK} is the identification of a holonomy perturbation $\pi$ so that $R^\nat_\pi(Y_1,T_1)$ is a circle and the restriction $R^\nat_\pi(Y_1,T_1)\to R(S^2, \{a,b,c,d\})$  an immersion with one double point,  illustrated in Figure \ref{pert}. 

The set of points in $R(Y_0,T_0)$ whose image in the pillowcase  intersect this circle  is a set of generators for the reduced instanton homology \cite{KM-khovanov}   of the knot obtained by gluing $Y_0$ to $Y_1$ and $T_0$ to $T_1$ (when $R(Y_0,T_0)$ is regular and transverse to $\{\gamma=\theta\}$).

Notice that whenever the restriction $R(Y_0,T_0)\to R(S^2, \{a,b,c,d\})$ is (locally) an immersion of a curve transverse to the arc $\{\gamma=\theta\}$, each intersection point with $\{\gamma=\theta\}$ gives rise to two intersection points with the  circle $R(Y_0,T_0)$, these two generators have relative $\ZZ/4$ grading $1$ in the instanton chain complex \cite{HHK}.

 The space $R(Y_0,T_0)$ is typically a singular 1 dimensional real analytic variety, homeomorphic to a finite 1-complex. The singular points (i.e. the vertices) avoid the arc $\{\gamma=\theta\}$, and the smooth locus is immersed, transverse to this arc.   In general the results of \cite{Herald-thesis} show that perturbations $\pi$ along curves in $Y_0\setminus T_0$  can always be found so that the perturbed space $R_\pi(Y_0,T_0)$ has these properties.

   \begin{figure}
\begin{center}
 \includegraphics[width=250pt]{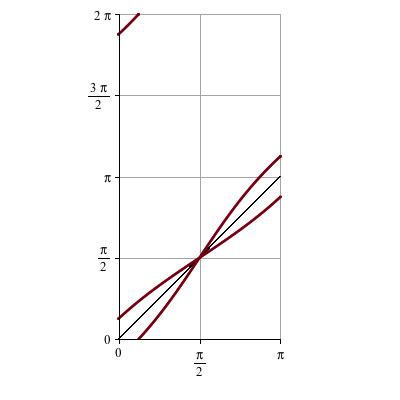}
 \caption{ \label{pert}}
\end{center}
\end{figure}

\subsection{Action of characters} The following  provides a convenient mechanism to simplify the analysis of the restriction map $R(Y_0,T_0)\to R(S^2, \{a,b,c,d\})$.

Given a character $\chi:\pi_1(Y_0\setminus T_0)\to\{\pm1\}$ and a traceless representation $\rho:\pi_1(Y_0\setminus T_0)\to SU(2)$, the product $\chi\cdot \rho$ is again a representation, since $\{\pm 1\}$ is the center of $SU(2)$.
Because $H_1(Y_0\setminus T_0)=\ZZ^2$ (generated by the meridians $a,b$), this gives a $\left(\ZZ/2\right)^2$ action on $R(Y_0,T_0)$. 

Similarly, the generators $a,b,c$ give an action of $\left(\ZZ/2\right)^3$ on $R(S^2,\{a,b,c,d\})$. The character $\chi:\pi_1(S^2\setminus \{a,b,c,d\})\to\{\pm1\}$ taking $a,b,c$ (and hence $d$) to $-1$ acts trivially on $R(S^2, \{a,b,c,d\})$, since its action is achieved by conjugation by $\bbk$ on the representation of Equation (\ref{eqpil}). Therefore,  the $\left(\ZZ/2\right)^3$ action descends to a $\left(\ZZ/2\right)^2$ action. This is generated by $(\gamma,\theta)\mapsto (\pi-\gamma,2\pi-\theta)$, corresponding to $\chi_1(a)=1=\chi_1(c), \chi_1(b)=-1$, and $(\gamma,\theta)\mapsto (\pi-\gamma, \theta+\pi)$, corresponding to $\chi_2(a)=1, \chi_2(b)=-1= \chi_2(c)$.  

The character $\chi_1$ extends to $\pi_1(Y_0\setminus T_0)$ since $a$ and $c$ are both sent to $1$.   The character sending $a,b,c,d$ to $-1$ extends to $\pi_1(Y_0\setminus T_0)$ and need not act trivially on $R(Y_0,T_0)$
This shows the following.

\begin{prop} \label{symmetry}The image of the restriction map $R(Y_0,T_0)\to R(S^2, \{a,b,c,d\})$  is invariant under the involution $(\gamma,\theta)\mapsto (\pi-\gamma,2\pi-\theta)$.  Moreover, multiplication by the character on $\pi_1(Y_0\setminus T_0)$ which sends every generator to $-1$ induces an involution on $R(Y_0,T_0)$ whose orbits are mapped to the same point by the restriction $R(Y_0,T_0)\to R(S^2, \{a,b,c,d\})$.\qed 
\end{prop}
An illustration of Proposition \ref{symmetry} is given in Figure \ref{fig45}, where the $\left(\ZZ/2\right)^2$ action on $R(Y_0,T_0)$ is given by reflections through the horizontal and vertical axes.

\section{Traceless binary dihedral representations of a two stranded tangle}

In this section  $(Y_0,T_0)$ is an oriented  2-stranded tangle with $Y_0$ a  $\ZZ/2$ homology 3-ball. We will identify the binary dihedral part of $R(Y_0,T_0)$   and determine its image in the pillowcase.

Define the {\em binary dihedral group} to be the subgroup 
$$B=\{\cos a+\sin a\bbk\}\cup \{\cos b\bbi+\sin b\bbj\}\subset SU(2).$$

\begin{df} \label{tbddef} We call a representation
 $\rho:\pi_1(Y_0\setminus T_0)\to SU(2)$ 
 {\em traceless binary dihedral} if its image lies in $B$
  and takes all meridians  into  the component $\{\cos b\bbi+\sin b\bbj\}$.  Denote by $R^{tbd}(Y_0,T_0)\subset R(Y_0,T_0)$  the image of all such representations in the space  of all $SU(2)$ conjugacy classes of traceless representations.  Denote its complement by $R^n(T_0,Y_0)$, thus 
\begin{equation}\label{deco}  R(Y_0,T_0)=R^{tbd}(Y_0,T_0)\sqcup R^{n}(Y_0,T_0).
\end{equation}
 \end{df}

We point out one subtlety of this definition:  not all traceless $SU(2)$ representations with image in $B$ are conjugate in $SU(2)$ to traceless binary dihedral representations as we have defined them.  The important feature of the definition is that the composite of a traceless binary dihedral representation $\rho$ with the continuous surjection $B\to \ZZ/2$ takes every meridian of $T_0\subset Y_0$ to the non-trivial element $1\in \ZZ/2$.   However, if $(Y_0,T_0)$ is obtained from a knot   $K$ in a $\ZZ/2$ homology sphere $Y$ by removing a tangle $(Y_1,T_1)$, then all  traceless representations of $\pi_1(Y\setminus K)$ with image in $B$ restrict to traceless binary dihedral representations in the sense of Definition \ref{tbddef}.

\subsection{The 2-fold branched cover of $(Y_0,T_0)$}
 Fix an identification $\partial(Y_0,T_0)\cong (S^2,\{a,b,c,d\})$  so that the two endpoints of one strand correspond to the labels $a,c$ and the other strand to $b,d$, as in Figures \ref{torusknot} and \ref{pretzel2}.  Note that $H_1(Y_0\setminus T_0)\cong \ZZ\oplus\ZZ$, generated by the meridians $a=c$ and $b=d$.
 
Denote by $\alpha:\pi_1(Y_0\setminus T_0)\to \ZZ$ the homomorphism taking each oriented meridian  to $1$, and by $\bar\alpha:\pi_1(Y_0\setminus T_0)\to \ZZ/2$
 its reduction modulo 2.  Thus $\alpha(a)=\alpha(b)=\alpha(c)=\alpha(d)=1\in \ZZ$. Let  $f:(\tilde Y_0,\tilde T_0)\to (Y_0,T_0)$ denote the corresponding 2-fold branched cover.

The boundary of $\tilde Y_0$ is a torus.   Choose a base point on $\partial\tilde Y_0$ which lies over the base point of $S^2\setminus\{a,b,c,d\}$.   Lift the   two loops $bc^{-1}$ and $ba$ in $\pi_1(S^2\setminus\{a,b,c,d\})$ to $\pi_1(\partial \tilde Y_0\setminus \partial \tilde T_0)$; these form a symplectic basis of $H_1(\partial\tilde Y_0)$.

The first homology  $H_1(\tilde Y_0)$ is isomorphic to $ \ZZ\oplus A$ where $A$ is an odd order finite abelian group.  
 One way to see this is to first note that the image of the restriction
$H^1(\tilde Y_0;\RR)\to H^1(\partial\tilde Y_0;\RR)\cong\RR\oplus \RR$ is a half-dimensional subspace by the usual argument using Poincar\'e duality and the universal coefficient theorem. Hence   $H^1(\tilde Y_0;\ZZ)$ has rank at least one. Now
glue the trivial tangle $(Y_1,T_1)$ to $(Y_0,T_0)$ to obtain a knot $K$ in a $\ZZ/2$ homology sphere $Y$ and apply the argument of Casson-Gordon 
\cite[Lemma 2]{casson-gordon} to conclude that the 2-fold branched cover $\tilde Y$ is a $\ZZ/2$ homology 3-sphere, and hence $H_1(\tilde Y;\ZZ/2)=0$ and so $H^1(\tilde Y;\ZZ)=0$. Since $\tilde Y_1$ is a solid torus, the $\ZZ/2$ homology 3-sphere $\tilde Y$ is obtained by Dehn filling $\tilde Y_0$, which reduces the rank of $H^1(\tilde Y_0;\ZZ)$ by at most one.  Hence $H^1(\tilde Y_0)$ has rank exactly one and $H_1(\tilde Y_0)$ is isomorphic to $ \ZZ\oplus A$ where $A$ is an odd order finite abelian group.

The covering involution $\tau$ acts as multiplication by $-1$ on $H^1(\tilde Y_0)$. Indeed, given $x\in H_1(\tilde Y_0;\ZZ)$, then $\tau (x)+x$ is fixed by $\tau$. Since $Y_0$ is a $\ZZ/2$ homology ball, it is a $\QQ$ homology ball, and Theorem III.2.4 of \cite{bredon} shows that 
 $\tau (x)+x$ maps to zero    in  $H_1(\tilde Y_0; \QQ)$.  Thus $\tau:H_1(\tilde Y_0;\QQ)\to H_1(\tilde Y_0;\QQ)$  is given by multiplication by $-1$, and by the universal coefficient theorem this is true for $H^1(\tilde Y_0;\ZZ)$.
 
 The covering involution  preserves the torsion subgroup $A\subset H_1(\tilde Y_0)$. Since $A$ has odd order, $A$ splits as the direct sum of the positive and negative eigenspaces of $\tau$, 
 $$A=A_+\oplus A_-.$$
If $x\in H_1(\tilde Y_0)$ is any element of infinite order, then $\tau(x)=-x + a_+$ for some $a_+$ in the odd torsion group $   A_+$. Thus replacing $x$ by $x-\frac{1}{2}a_+$ we obtain an element $x$ of infinite order satisfying $\tau(x)=-x$.  Hence we obtain a decomposition (where $\ZZ_-$ denotes the integers with the non-trivial $\ZZ/2$ action)
$$H_1(\tilde Y_0)\cong H_1(\tilde Y_0)_+\oplus H_1(\tilde Y_0)_-=A_+\oplus (A_-\oplus \ZZ_-).$$

A similar argument applies to the boundary torus $\partial \tilde Y_0 $. In this case  $H_1(\partial Y_0;\ZZ)=0$, hence the covering involution acts as multiplication by $-1$ and so $H_1(\partial \tilde Y_0 )=H_1(\partial \tilde Y_0)_-$.   In particular, since the loops $bc^{-1}$ and $ba$ in $\pi_1(S^2\setminus\{a,b,c,d\})$ lift to a symplectic basis of $H_1(\partial(\tilde Y_0))$, the  classes they represent in $H_1(\partial(\tilde Y_0))$ satisfy $\tau([bc^{-1}])=-[bc^{-1}]$ and $\tau([ba])=-[ba]$.

In the special case when $Y_0$ is a $\ZZ$ homology ball, then Theorem III.2.4 of \cite{bredon} implies that  $H_1(\tilde Y_0)=H_1(\tilde Y_0)_-$ and $A=A_-$.  In particular if $(Y_0,T_0)$ is obtained from a knot in a homology 3-sphere by removing a tangle in a 3-ball, then $H_1(Y_0)=H_1(Y_0)_-$ and $A=A_-$.

\subsection{A slope associated to a 2-stranded tangle with marked boundary}\label{bdprel}

Choose one of the two generators $h\in H^1(\tilde Y_0)$ and consider it as a homomorphism $h:\pi_1(\tilde Y_0)\to \ZZ$.
Evaluating $h$ on the symplectic basis for $H_1(\partial \tilde Y_0)$ gives a pair of integers: 
$$
h(bc^{-1})\text{~and~} h(ba).
$$

Changing the sign of $h$ or  choosing  the other base point for $\partial \tilde Y_0$ changes the signs of both $h(bc^{-1})$ and $h(ba)$.  Moreover, $h(bc^{-1})$ and $h(ba)$  are not both zero since the image $H^1(\tilde Y_0)\to H^1(\partial \tilde Y_0)$ is non-trivial. 
Therefore ratio $h(ba)/h(bc^{-1})$ is a  well defined invariant of the pair $(Y_0,T_0)$ equipped with a marking of its boundary.

\subsection{Determination of binary dihedral representations}

In preparation for the statement of the following theorem, we summarize the assumptions, notation, and calculations form the previous sections: $Y_0$ is a $\ZZ/2$ homology 3-ball containing an oriented  2-stranded tangle $T_0$. 
We denote by $\alpha:\pi_1(Y_0\setminus T_0)\to \ZZ$ the unique homomorphism taking all oriented meridians of $T_0$ to $1$.  An identification $\partial (Y_0,T_0)\cong (S^2,\{a,b,c,d\})$ is given so that one component of $T_0$ meets the boundary 2-sphere in the points labelled $a,c$ and the other in the points labelled $b,d$. The same labels are used for the corresponding meridians of $S^2\setminus \{a,b,c,d\}$ and $Y_0\setminus T_0$.

The 2-fold branched cover of $Y_0$ branched over $T_0$ associated to the mod 2 reduction $\bar\alpha$ of $\alpha$ is denoted by  $\tilde Y_0$. The homology $H_1(\tilde Y_0)$ is isomorphic to $\ZZ\oplus A$ for $A$ an odd torsion group, and $h:
H_1(\tilde Y_0)\to \ZZ$ generates $H^1(\tilde Y_0)$.  The loops $bc^{-1}$ and $ba$ lift to $\tilde Y_0$.  The covering involution $\tau:\tilde Y_0\to \tilde Y_0$ splits $H_1(\tilde Y_0)$ and its torsion subgroup $A$ into  $\pm1$ eigenspaces, 
$$H_1(\tilde Y_0)=H_1(\tilde Y_0)_+\oplus H_1(\tilde Y_0)_-\cong A_+\oplus (A_-\oplus \ZZ_-).$$ In particular, $\tau^*(h)=-h$.

\medskip

 Note that $A_-$ and $\Hom(A_-,S^1)$ are isomorphic.    Consider the equivalence relation on $\Hom(A_-,S^1)$ which identifies a character $\ell:A_-\to S^1$ with its complex conjugate $\ell^{-1}$.  Since $A_-$ has odd order, the only fixed point of this involution is the trivial character.

\begin{thm}\label{bd}  The subspace   $R^{tbd}(Y_0,T_0)\subset R(Y_0,T_0)$ consisting of $SU(2)$ conjugacy classes of traceless binary dihedral representations is homeomorphic to the disjoint union of one arc and $\frac{n-1}2$ circles, where $n$ is the order of $A_-$. The path components are indexed by the equivalence classes in  $\Hom(A_-,S^1)/_{\ell \sim\ell^{-1}}$.  The arc corresponds to the trivial character, and the circles correspond to the non-trivial characters. 

The arc corresponding to the trivial character maps to  the pillowcase  by the map
\begin{equation*}
\label{bdpip}
t\in [0,\pi]\mapsto (\gamma,\theta)=(h(ba)t, (h(ba)-h(bc^{-1}))t)
\end{equation*}
where $h\in \Hom(\pi_1(\tilde Y_0),\ZZ)=H^1(\tilde Y_0)\cong \ZZ$ denotes the generator.

The circle corresponding to the non-trivial character $\ell$ maps to the pillowcase by the   map
\begin{equation*}
\label{bdpip2}
e^{t\bbi}\in S^1\mapsto(\gamma,\theta)=(h(ba)t, (h(ba)-h(bc^{-1}))t)-(\log \ell(a^{-1}b), \log\ell(a^{-1}c)).
\end{equation*}
\end{thm}

\begin{proof}

To keep notation compact, we write $S^1= \{e^{x\bbk}~|~x\in \RR\}$ so that   the binary dihedral subgroup $B$ is the union $S^1\cup S^1\bbi$.   We  denote again by $h$  its composite with the Hurewicz map and the inclusion
$$h:\pi_1(\tilde Y_0\setminus \tilde T_0)\to 
H_1(\tilde Y_0\setminus \tilde T_0)\to H_1(\tilde Y_0)\xrightarrow{h} \ZZ.$$
Similarly, given  $\ell\in \Hom(H_1(\tilde Y_0), S^1)$, abuse notation  and also denote by $\ell $ the composition $$\pi_1(\tilde Y_0\setminus \tilde T_0)\to H_1(\tilde Y_0\setminus \tilde T_0)\to H_1(\tilde Y_0)\to S^1.$$

We introduce the notation 
$$X^{tbd}(Y_0,T_0)=\{\rho:\pi_1(Y_0\setminus T_0)\to B~|~ \rho(a)=\bbi, \rho(\mu)\in S^1\bbi ~\text{ for each meridian~}\mu\}$$
for the set of traceless binary dihedral representations satisfying $\rho(a)=\bbi$.   Any traceless binary dihedral representation $\rho:\pi_1(Y_0\setminus T_0)\to B$ may be  conjugated by  $e^{x\bbk}$ if necessary so that $\rho(a)=\bbi$.  Thus the natural map $X^{tbd}(Y_0,T_0)\to R^{tbd}(Y_0,T_0)$ is surjective.

\medskip

The action of the covering transformation $\tau$ on on $\pi_1(\tilde Y_0\setminus \tilde T_0)$, viewed as an index 2 subgroup of $\pi_1(Y_0\setminus T_0)$,  is given by conjugation by $a$.  Thus if $i:\pi_1(\tilde Y_0\setminus \tilde T_0)\to H_1(\tilde Y_0)$ denotes the composite of the inclusion and the Hurewicz map, $i(a\beta a^{-1})=\tau_*(i(\beta))$.  

For any $\beta\in \pi_1( Y_0\setminus  T_0)$, the product
$a^{-\alpha(\beta)}\beta$ lies in the subgroup $\pi_1(\tilde Y_0\setminus \tilde T_0)$, since 
$$\bar \alpha(a^{-\alpha(\beta)}\beta)=\bar \alpha(a^{-\alpha(\beta)})+\bar \alpha( \beta)=-\bar\alpha(\beta)+\bar \alpha( \beta)=0.
$$

\medskip

Given  $\ell\in \Hom(H_1(\tilde Y_0)/A_+, S^1)$, define the function:
 \begin{equation}\label{rhoell}
\rho_\ell:\pi_1(Y_0\setminus T_0)\to B,
~\rho_\ell(\beta)=\bbi^{\alpha(\beta)}\ell(a^{-\alpha(\beta)}\beta)
\end{equation}
Then,  since $\ell$ vanishes on $A_+$, $\ell(a\beta a^{-1})=\tau^*(\ell)(\beta)=\ell(\beta)^{-1}$,
\begin{eqnarray*}
\ell(a^{-\alpha(\beta_1\beta_2)}\beta_1\beta_2)&=&\ell(a^{-\alpha(\beta_2)}a^{-\al(\beta_1)}\beta_1a^{\alpha(\beta_2)}a^{-\alpha(\beta_2)}\beta_2)\\
&=&\ell(a^{-\alpha(\beta_2)} a^{-\al(\beta_1)}\beta_1 a^{\alpha(\beta_2)})\ell(a^{-\alpha(\beta_2)}\beta_2 )\\
&=&
\ell(a^{-\al(\beta_1)}\beta_1)^{(-1)^{-\alpha(\beta_2)}}\ell(a^{-\alpha(\beta_2)}\beta_2 )
\end{eqnarray*}
and so
\begin{eqnarray*}\rho_\ell(\beta_1\beta_2)&=&\bbi^{\alpha(\beta_1\beta_2)}
\ell(a^{-\alpha(\beta_1\beta_2)}\beta_1\beta_2)   \\
&=&\bbi^{\alpha(\beta_1)}\bbi^{\alpha(\beta_2)}\ell(a^{-\al(\beta_1)}\beta_1)^{(-1)^{-\alpha(\beta_2)}}\ell(a^{-\alpha(\beta_2)}\beta_2 )
\\
&=&\bbi^{\alpha(\beta_1)}\ell(a^{-\al(\beta_1)}\beta_1)\bbi^{\alpha(\beta_2)} \ell(a^{-\alpha(\beta_2)}\beta_2 )
=\rho_\ell(\beta_1)\rho_\ell(\beta_2).
\end{eqnarray*}
Hence $\rho_\ell$ is a homomorphism.  It satisfies 
$$\rho_\ell(a)=\bbi, ~\rho_\ell(b)=\bbi\ell(a^{-1}b), ~\rho_\ell(c)=\bbi\ell(a^{-1}c), ~\rho_\ell(d)=\bbi\ell(a^{-1}d).$$
If $\mu\in \pi_1(Y_0\setminus T_0)$ denotes any oriented meridian  of $T_0$,
$\alpha(\mu)=1$, so that $\rho(\mu)=\bbi \ell(a^{-1}\mu)$, which lies in $S^1\bbi$. Hence $\rho_\ell$ is a traceless binary dihedral representation. 
 
 \medskip

  Given and $\rho\in X^{tbd}(Y_0,T_0)$, the function 
 $$\ell_\rho:\pi_1(\tilde Y_0\setminus \tilde T_0)\to S^1,~ \ell_\rho(\beta)=\bbi^{-\alpha(\beta)}\rho(\beta)$$
is a homomorphism since $\bbi^{-\alpha(\beta)}=\pm 1$ for $\beta\in \pi_1(\tilde Y_0\setminus \tilde T_0)$. Moreover, $$\ell_\rho(\mu^2)=\bbi^{-2}\rho(\mu)^2=(-1)(-1)=1$$ for any meridian $\mu$ since $\rho$ is traceless binary dihedral. Hence
$\ell_\rho$ factors through $H_1(\tilde Y_0)$.

Suppose that $\beta\in H_1(\tilde Y_0)$ satisfies $\tau(\beta)=\beta$, i.e.~  $\beta\in A_+$.  Then $\bar\alpha(\beta)=0$, so that  $\rho(\beta)\in S^1$. Hence
 \begin{eqnarray*}
\ell_\rho(\beta)&=& \ell_\rho(a\beta  a^{-1})=\bbi^{-\alpha(a\beta a^{-1})}\rho(a\beta a^{-1})
 \bbi^{-\alpha(\beta)}\bbi \rho(\beta)(-\bbi)\\&=& \bbi^{-\alpha(\beta)}  \rho(\beta)^{-1}
=\bbi^{-\alpha(\beta)} \ell_\rho(\beta)^{-1}\bbi^{\alpha(\beta)}=\ell_\rho(\beta)^{-1}.
\end{eqnarray*}
Thus $\ell_\rho(\beta)=\pm1$. But since $\beta\in A_+$, an odd torsion group, this implies that $\ell_\rho(\beta)=1$. Therefore,  $\ell_\rho$ factors through  $H_1(\tilde Y_0)/A_+$.

  \medskip

Since $\ell(a^2)=1$ and $\alpha(\beta)$ is even for $\beta\in \pi_1(\tilde Y_0\setminus \tilde T_0)$,  
$$ \ell\rho_{\ell}(\beta)= \bbi^{-\alpha(\beta)}\bbi^{\alpha(\beta)}\ell(a^{-\alpha(\beta)}\beta)=\ell(a^{2n})\ell(\beta)=\ell(\beta).$$ 
Hence the composite  $\ell\mapsto \ell\rho_{\ell}$ equals the identity. 

If $\rho_1,\rho_2\in X^{tbd}(Y_0,T_0)$ satisfy $\ell_{\rho_1}=\ell_{\rho_2}$, then the restrictions of $\rho_1,\rho_2$  to $\pi_1(\tilde Y_0\setminus \tilde T_0)\to S^1$ are equal. But since the composites  of $\rho_1$ and $\rho_2$ with the surjection $B\to \ZZ/2$ are equal, it follows that $\rho_1$ and $\rho_2$ are equal.

We can summarize what has been established so far as follows. Taking conjugacy classes gives a surjection $X^{tbd}(Y_0,T_0)\to R^{tbd}(Y_0,T_0)$ and the maps $\rho\mapsto \ell_\rho$  and $\ell\mapsto \ell_{\rho_\ell}$ define inverse bijections
 \begin{equation}
\label{eq3.0}
X^{tbd}(Y_0,T_0)\cong  \Hom(H_1(\tilde Y_0)/A_+, S^1).
\end{equation}

It remains to analyze the effect of conjugation.
Suppose   that $\rho_1,\rho_2\in X^{tbd}(Y_0,T_0)$ are two representations   which map to the same point in $R^{tbd}(Y_0,T_0)$.   Since $\rho_i(a)=\bbi$, $\rho_2= e^{x\bbi}\rho_2 e^{-x\bbi}$ for some $x$.  Now,
$$B\cap e^{x\bbi}B e^{-x\bbi}=
\begin{cases} B &\text{ if } x\equiv 0\mod \frac\pi 2\\
 \{\pm 1,\pm\bbi,\pm\bbj,\pm\bbk\}& \text{ if } x\equiv \frac\pi 4\mod \frac\pi 2\\
 \{\pm 1,\pm\bbi\}&\text{~otherwise}.
\end{cases}
$$
We treat these cases one at a time and show that in each case, either $\rho_2=\rho_1$, or $\rho_2=\bbi \rho_1(-\bbi)$.  

For the first case, if $x\equiv 0\mod \pi$, then $\rho_1=\rho_2$. If $x\equiv \frac\pi 2\mod \pi$, then $\rho_2=\bbi \rho_1(-\bbi)$.
For the third case, $\rho_1$ and $\rho_2$ take their values in $B\cap e^{x\bbi}B e^{-x\bbi}=
 \{\pm 1,\pm\bbi\}=  e^{-x\bbi}B e^{x\bbi}\cap B$, which is fixed by conjugation by $e^{-x\bbi}$. Hence $\rho_2=\rho_1$.

Consider, then the second case. So $e^{x\bbi}= e^{\frac{n\pi}{4}\bbi}$ for some odd integer $n$.  Suppose that $\rho_2\ne \rho_1$. Choose $\beta_0\in \pi_1(Y_0\setminus T_0)$ so that $\rho_2(\beta_0)\ne\rho_1(\beta_0)$. Replacing  $\beta_0$ by $a\beta_0$ if necessary,  and noting that $\rho_1(a)=\bbi=\rho_2(a)$, we may assume that $\beta_0$ lies in  the subgroup $ \pi_1(\tilde Y_0\setminus \tilde T_0)$.  Since $\rho_1$ and $\rho_2$ take their values in $B\cap e^{x\bbi}B e^{-x\bbi}=
 \{\pm 1,\pm\bbi,\pm\bbj, \pm \bbk\}$ and $\bar \alpha(\beta_0)=0$, it follows that $\rho_2(\beta_0)=\pm \bbk$. Similarly $\rho_1(\beta_0)=\mp \bbk$ and so by reindexing if needed, 
$$\rho_1(\beta_0)=  \bbk, \rho_2(\beta_0)=-\bbk.$$

Given any $\beta\in \pi_1(\tilde Y_0\setminus \tilde T_0)$,
$$\ell_{\rho_1}(\beta)=(-1)^{-\alpha(\beta)/2}\rho_1(\beta).$$
The homomorphism $\ell_{\rho_1}$ takes values in the order 4 cyclic subgroup of $\{\pm 1,\pm\bbi,\pm\bbj, \pm \bbk\}$generated by $\bbk$, and factors through $H_1(\tilde Y_0)/A_+$ which is isomorphic to $\ZZ\oplus A_-$ where $A_-$ is an odd torsion group. Thus $A_-$  and hence $\ell_{\rho_1}$  factors through the generator $h:H_1(\tilde Y_0)\to \ZZ$ of $H^1(\tilde Y_0)\cong \ZZ$.  This implies that
$$\ell_{\rho_1}(\beta)= e^{ {m_1\pi} h(\beta)\bbk/2}$$ for some integer $m_1$. Since $\rho_1=\rho_{\ell_{\rho_1}}$, this in turn implies that  for any $\beta\in \pi_1(Y_0\setminus T_0)$,
\begin{equation}
\label{eq3.1}\rho_1(\beta)=\bbi^{\alpha(\beta)}e^{ {m_1\pi} h(a^{-\alpha(\beta)}\beta)\bbk/2}. 
\end{equation}

Similarly
\begin{equation}
\label{eq3.2}
\rho_2(\beta)=\bbi^{\alpha(\beta)}e^{ {m_2\pi} h(a^{-\alpha(\beta)}\beta)\bbk/2} 
\end{equation}   for some integer $m_2$. Since $\rho_1(\beta_0)=\bbk=-\rho_2(\beta_0)$, 
we see that $m_1$ and $m_2$ are odd, and that $ m_2\equiv -m_1 \mod 4$.
But then using Equation (\ref{eq3.1}) and (\ref{eq3.2}) we see that  $\rho_2=\bbi\rho_1(-\bbi)$.

Thus we have shown that in all three cases, if $\rho_1,\rho_2\in X^{tbd}(Y_0,T_0)$  map to the same point in $R(Y_0,T_0)$, then either $\rho_1=\rho_2$ or $\rho_2=\bbi\rho_1(-\bbi)$.

\medskip
If $\rho_2=\bbi\rho_1(-\bbi)$, then 
$$\ell_{\rho_2}(\beta)=\bbi^{-\alpha(\beta)}\bbi\rho_1(\beta)(-\bbi)=
\bbi^{-\alpha(\beta)}\bbi(
\bbi^{\alpha(\beta)}\ell_{\rho_1}(\beta))(-\bbi)=\ell^{-1}_{\rho_1}(\beta).
$$
 Hence the bijection of Equation (\ref{eq3.0}) sets up identifications
 $$\Hom(H_1(\tilde Y_0)/A_+,S^1)/_{\ell\sim \ell^{-1}} \cong X^{tbd}(Y_0,T_0)/_{\rho\sim \bbi \rho(-\bbi)}\cong R^{tbd}(Y_0,T_0)\subset R(Y_0,T_0).$$

 \medskip

Define an action of $S^1$ on $X^{tbd}(Y_0,T_0)$ by 
$$(e^{t\bbk}\cdot \rho)(\beta) = \rho(\beta)e^{th(a^{-\alpha(\beta)}\beta)\bbk}, ~ \beta\in \pi_1(Y_0\setminus T_0).
$$
The action is free since $h$ is takes the value $1$ on some $\beta\in H_1(\tilde Y_0)$.
The action induces the free $S^1$ action on $\Hom(H_1(\tilde Y_0)/A_+,S^1)$ given by
$$(e^{t\bbk}\cdot \ell)(\beta) = \ell(\beta)e^{th( \beta)\bbk},~ \beta\in H_1(\tilde Y_0).
$$

  Restricting to the torsion subgroup $A_-\subset H_1(\tilde Y_0)$ determines a map $$\Hom(H_1(\tilde Y_0)/A_+,S^1)\to \Hom(A_-,S^1).$$ The fibers of this map are exactly the    orbits of the $S^1$ action, since $h:H_1(\tilde Y_0)\to \ZZ$ vanishes on $A_-$. The restriction map is equivariant with respect to the involution $\ell\mapsto\ell^{-1}$.  
  
  Since $A_-$ has odd order, the involution $\ell\to \ell^{-1}$ acts on $\Hom(A_-,S^1)$ with precisely  one fixed point, namely the trivial homomorphism $\ell\equiv 1$.    The involution on the fiber over the trivial homomorphism corresponds to $h\mapsto -h$, and   hence has quotient an arc $[0,\pi]$, with $t$ corresponding to $\beta\mapsto \bbi^{\alpha(\beta)}e^{th(a^{-\alpha(\beta)}\beta)\bbk}$.
  Therefore, the quotient
$\Hom(H_1(\tilde Y_0)/A_+,S^1)/_{\ell\sim \ell^{-1}}$ is homeomorphic to one arc and $\frac{n-1}{2}$ circles, where $n$ is the order of $A_-$. The circles are indexed by the non-trivial characters on $A_-$, up to $\ell\sim \ell^{-1}$. 
  This   establishes the first part of the statement of Theorem \ref{bd}.

\medskip

Define integers $p,q$ by $$p=h(bc^{-1})\text{~and ~}q=h(ba).$$

Note that $a^2,b^2,c^2,$ and $d^2$ are in $ \pi_1(\tilde Y_0\setminus \tilde T_0)$ and in the kernel of $h$, since these loops are meridians to the branch set $\tilde T_0$.  One computes $(e^{t\bbk}\cdot \rho)(a)=\bbi$ and 
$$(e^{t\bbk}\cdot \rho)(b)=\rho(b)e^{th(a^{-1}b)\bbk}=
\rho(b)e^{-th(ba^{-1})\bbk}=\rho(b)e^{-(h(ba)+h(a^{-2}))t\bbk}=\rho(b) e^{-qt\bbk}=e^{qt\bbk}\rho(b)$$
since $h(a^2)=0$, $\rho(b)\in S^1\bbi$, and $h(ba)=q$.   Then since $\alpha(cb^{-1})=0$,
 \begin{eqnarray*}
(e^{t\bbk}\cdot \rho)(c)&=&(e^{t\bbk}\cdot \rho)(cb^{-1})(e^{t\bbk}\cdot \rho)(b) 
=\rho(cb^{-1})e^{th(cb^{-1})\bbk} e^{qt\bbk}\rho(b)
=
\rho(cb^{-1})e^{-pt \bbk} e^{qt\bbk}\rho(b)\\&=&\rho(cb^{-1})e^{(q-p)t \bbk} \rho(b)
=e^{(q-p)t \bbk} \rho(cb^{-1})\rho(b)=e^{(q-p)t \bbk} \rho(c).
\end{eqnarray*}
and similarly $(e^{t\bbk}\cdot \rho)(d)=e^{-pt\bbk}\rho(d)$.

Note that $\rho(b)=\rho(a)\ell_\rho(a^{-1}b)=\bbi\ell_\rho(a^{-1}b)$, and likewise 
$\rho(c)=\bbi\ell_\rho(a^{-1}c),~\rho(d)=\bbi\ell_\rho(a^{-1}d)$. Hence
\begin{eqnarray*}
(e^{t\bbk}\cdot \rho)(a)&=&\bbi,\\
(e^{t\bbk}\cdot \rho)(b)&=&e^{qt\bbk}\bbi\ell_\rho(a^{-1}b)=e^{(qt- \log(\ell_\rho(a^{-1}b))\bbk}\bbi\\
(e^{t\bbk}\cdot \rho)(c)&=& e^{(q-p)t\bbk}\bbi\ell_\rho(a^{-1}c)=e^{((q-p)t- \log(\ell_\rho(a^{-1}c))\bbk}\bbi,
\end{eqnarray*}
 and so the second part of the statement of Theorem \ref{bd} is proved. 
\end{proof}

The arc of traceless binary dihedral representations produced in Theorem \ref{bd} coincides with the arcs found for the tangles corresponding to 2-bridge knots in \cite[Theorem 10.2]{HHK} and by an ad-hoc argument for torus knots in  \cite[Proposition 11.2]{HHK}.

One simple way to produce examples with many binary dihedral components is to take the connected sum of $Y_0$ with an odd order lens space; this replaces $A_-$ by $A_-\oplus \ZZ/k$ for some odd $k$, and hence the number of circles in $R^{tbd}(Y_0,T_0)$  increases by $n\frac{k-1}{2}$, where $n$ denotes the order of $A_-$. But one can also increase the number of binary dihedral components for any tangle $(Y_0,T_0)$ without changing $Y_0$ as follows.  Let $T'_0$ be obtained from $T_0$ by taking the connected sum of $T_0$ with a small trefoil knot.    Then $\tilde Y_0$ is replaced by the connected sum $\tilde Y_0\# L(3,1)$, since the 2-fold branched cover of a trefoil is the lens space $L(3,1)$. Since $Y_0$ is a 3-ball, $A_-$ is replaced by $A'_-=A_-\oplus \ZZ/3$ and the number of circles of traceless binary dihedral representations increases  by $n$, where $n$ is the order of $A_-$.

 \section{torus knots}

\subsection{The tangle associated to a torus knot}  

We recall the analysis of $R(Y_0,T_0)$ for tangles associated to torus knots initiated  in \cite{HHK}.  Let $p,q$ be a relatively prime pair of integers and let $r,s$ be integers so that  $pr+qs=1$.

Figure \ref{torusknot} illustrates a $(p,q)$ torus knot $T_{p,q}$ in $S^3$. We view $S^3$ as $\frac{q}{r}$ and $-\frac{s}{p}$ Dehn surgery on the two components of a Hopf link.  The knot  $T_{p,q}$ is isotopic to a curve parallel to the first component. In the figure, $T_{p,q}$ has been isotoped so that it meets a 3-ball in a trivial tangle $(Y_1,T_1)$.

  \begin{figure}
\begin{center}

\def\svgwidth{3in}

 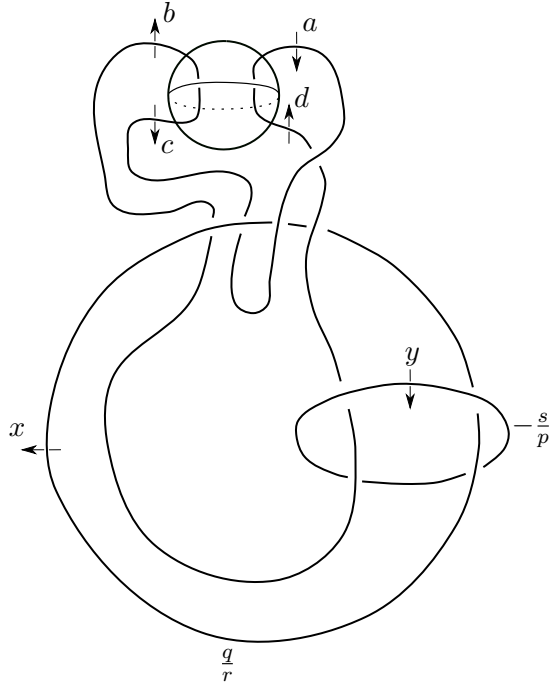
 \caption{The $(p,q)$ torus knot \label{torusknot}}
\end{center}
\end{figure} 
Let $Y_0$ be the exterior 3-ball, containing the 2-stranded tangle $T_0$.  
Consider $x,y,a,b,c,d$, the generators of $\pi_1(Y_0\setminus  T_0)$ illustrated in Figure \ref{torusknot}.
Set
 $$A=(xb)^qy^r \ \ \text{and}\  \ B=(d\bar a y)^{-s} x^p.$$
 Then $\pi_1(Y_0\setminus  T_0)$ is the free group on $A,B$ and 
 \begin{equation}
\label{words}
a=A^{s+p}B^{q-r}, ~b=B^{-r}A^s, ~ c=B^{-r}A^{s+p}B^q=B^{-r}aB^r,~
d=B^{-q}A^sB^{q-r}=B^{-(q-r)}bB^{q-r}.
\end{equation}

 Since $\pi_1(Y_0\setminus  T_0)$ is free on $A$ and $B$,
 the assignment $A\mapsto M, B\mapsto N$ gives a representation for any pair $M,N\in SU(2)$, but a general such assignment will not  yield a traceless representation.

\vfill\eject

\subsection{Binary dihedral representations for torus knots}

\begin{thm}\label{thm3.1}  For the tangle $(Y_0,T_0)$ obtained by removing a 3-ball from the $(p,q)$ torus knot as illustrated in Figure \ref{torusknot}, the subspace $R^{tbd}(Y_0,T_0)\subset R(Y_0,T_0)$ is an arc which maps to the pillowcase by 
$$ t\in [0,\pi]\mapsto  (\gamma,\theta)=  
 \begin{cases} (t,2t) &\text{ if } p,q,r \text{ are odd,}\\
 (t,0)&\text{ if } p,q  \text{ are odd and } r \text{ is even,}\\
((q-2r)t, -2rt) & \text{ if } p  \text{ is even and } q \text{ is odd,}\\
((2s+p)t, (2s+2p)t) & \text{ if } p  \text{ is odd and } q \text{ is even.}
 \end{cases}$$
\end{thm}

\begin{proof} 

To find the image of $R^{tbd}(Y_0,T_0)$ in the pillowcase, it is enough to find $H_1(\tilde{Y}_0)$, $h(ba)$, and $h(bc^{-1})$ by Theorem \ref{bd}. With this in mind, notice that $H_1(\tilde{Y}_0)$ is the quotient of $H_1(\tilde{Y}_0\setminus\tilde{T}_0)$ by the group generated by the homology class of the lifts of $a^2$ and $b^2$ to $\tilde{Y}_0\setminus\tilde{T}_0$. Now, $H_1(\tilde{Y}_0\setminus\tilde{T}_0)$ is the abelianization of $\pi_1(\tilde{Y}_0\setminus\tilde{T}_0)$ and this group is determined by the exact sequence:
$$1\to\pi_1(\tilde{Y}_0\setminus\tilde{T}_0)\to \pi_1({Y}_0\setminus{T}_0) \xrightarrow{\bar \alpha}  \ZZ/2\to 1.$$

As $\pi_1(Y_0\setminus T_0)$ is the free group $\langle A,\;B\rangle$, we can compute $\pi_1(\tilde{Y}_0\setminus\tilde{T}_0)$  and expressions for the lifts of loops in $\pi_1(Y_0\setminus T_0)$  by the Reidemeister-Schreier method, examining the 2-fold covering space of  $S^1\vee S^1$, a wedge of two circles  labelled $A,B$.

From Equation (\ref{words}) we get:
$$
1=(s+p)\bar{\alpha}(A)+(q-r)\bar{\alpha}(B) \text{~and~}
1=s\bar{\alpha}(A)-r\bar{\alpha}(B)
$$
 or equivalently
\begin{equation}\label{4.1}
0=p\bar{\alpha}(A)+q\bar{\alpha}(B)\text{~and~}
 1=s\bar{\alpha}(A)-r\bar{\alpha}(B).
\end{equation}

We thus see that $\bar\alpha(A)$ and $\bar\alpha(B)$ in $\ZZ/2$ are determined by the parities of $p,\,q,\,r$, and $s$.\\

If $p$ is even (and hence $q, s$ odd), Equation (\ref{4.1}) shows that $\bar{\alpha}(A)=1$ and $\bar{\alpha}(B)=0$. The corresponding 2-fold covering space of $S^1\vee S^1$ is illustrated in Figure \ref{cover}, where the lifts of  $A$ are the 1-cells $A_1,A_2$ and the lifts of   $B$ are the loops $B_1,B_2$.
\begin{figure}[h!]
\centering
\def\svgwidth{0.5\columnwidth}
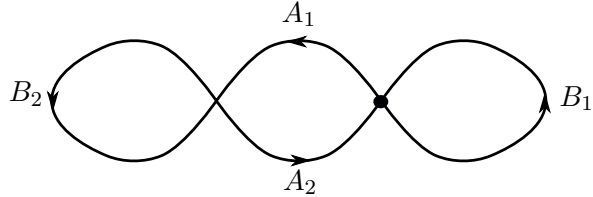
\caption{$\tilde{Y}_0\setminus\tilde{T}_0$ for $p$ even\label{cover}}
\end{figure}

The fundamental group $\pi_1(\tilde{Y}_0\setminus\tilde{T}_0)$ is free on $ x=B_1,\; y=A_1A_2,\; z=A_2^{-1}B_2A_2$.  By the Seifert-Van Kampen theorem,
$\pi_1(\tilde{Y}_0)$ is obtained by taking the quotient by the normal subgroup generated by $a^2$ and $b^2$.
Lifting $a^2$, $b^2$ to $\tilde {a} ^2,\tilde {b} ^2\in \pi_1(\tilde{Y}_0\setminus\tilde{T}_0)$ we find, using Equation (\ref{words}) and Figure \ref{cover}, that  $$
\tilde {a} ^2 =y^{(s+p+1)/2}z^{q-r}y^{(s+p-1)/2}x^{q-r}\text{~and~}
\tilde {b} ^2=x^{-r}y^{(s+1)/2}z^{-r}y^{(s-1)/2}.$$

This shows that $H_1(\tilde{Y}_0)$ has the abelian presentation
$$H_1(\tilde{Y}_0)=\langle x,\;y,\;z\:\mid\: (q-r)x+(s+p)y+(q-r)z=0,\: -rx+sy-rz=0\rangle$$ or equivalently $$H_1(\tilde{Y}_0)=\langle x,\;y,\;z\:\mid\: y=0,\: x+z=0\rangle=\langle x \rangle\cong \ZZ.$$
 
Let $h\in H^1(\tilde{Y}_0)$  be the generator satisfying $h(x)=1$. Then $h(z)=-1$ and $h(y)=0$. We compute,  using Equation (\ref{words}) and Figure \ref{cover}:
\begin{align*}
\widetilde{ba}&=x^{-r}y^{(2s+p)/2}x^{q-r}\\
\widetilde{bc^{-1}}&=x^{-r}y^{(s+1)/2}z^{-q}y^{-(s+p+1)/2}x^{r}
\end{align*}
and so $$h(ba)=-2r+q \text{ and } h(bc^{-1})=q.$$ Theorem \ref{bd} then implies that the arc of binary dihedral representations is given by $$\left(\gamma(t),\theta(t)\right)=((-2r+q)t, -2rt).$$
The other cases, when $p$ is odd,  are solved in a similar manner.
\end{proof}

\subsection{Non-binary dihedral representations of torus knots} We recall (and simplify and extend) the analysis of $R(Y_0,T_0)$  for the tangle of Figure \ref{torusknot} associated to a torus knot initiated in \cite{HHK}.

 For each integer $n$ denote by   $T_n(x)$ and $S_n(x)$ the (Chebyshev) polynomials satisfying   
$$
\cos(nu)=T_n(\cos u) \text{~ and ~}\sin(nu)=\sin uS_n(\cos u).
$$ and define the functions
$$p_1(x,y,\tau)=T_{s+p}(x)T_{q-r}(y)-\sqrt{(1-x^2)(1-y^2)} S_{s+p}(x)S_{q-r}(y)~ \tau, 
$$

$$p_2(x,y,\tau)=T_{s}(x)T_{-r}(y)-\sqrt{(1-x^2)(1-y^2)} S_{s}(x)S_{-r}(y)~\tau, 
$$
 Let $\arccos:[-1,1]\to [0,\pi]$ denote the usual branch of the inverse cosine.

The following theorem essentially identifies $R(Y_0,T_0)$ with the set of common zeros of $p_1$ and $p_2$ in the unit cube. 

 \begin{thm}[\cite{HHK}, Theorem 11.1] \label{thm2} Let  $W$ denote the set of common zeros of $p_1$ and $p_2$ in the unit cube:
 $$W=\{(x,y,\tau)~|~ p_1(x,y,\tau)=0=p_2(x,y,\tau), |x|\leq 1, |y|\leq 1,|\tau|\leq 1\}.$$
 Then the assignment
 \begin{equation}
\label{wfam}w=(x,y,\tau)\in W\to  \big(r_w(A)= e^{\arccos(x)\bbi},~r_w(B)= e^{\arccos(y)e^{\arccos(\tau)\bbk}\bbi}\big)
\end{equation}
determines a   surjection
$$r:W\to R(Y_0,T_0).$$
Given $w=(x_0,y_0,\tau_0)\in W$, the fiber over $r_w$ is the single point $\{w\}$ unless 
  $|x_0|=1$ or $|y_0|=1$, in which case the fiber is the arc $(x_0,y_0,\tau),\tau\in [-1,1]$.
  \end{thm}
We omit the proof, which can be found in \cite{HHK}, but note that the essential step is    to substitute $A=e^{\arccos(x)\bbi},~B= e^{\arccos(y)e^{\arccos(\tau)\bbk}\bbi}$   into the expressions $a=A^{s+p}B^{q-r}$ and $b=B^{-r}A^s$. Taking real parts gives the functions $p_1$ and $p_2$, which vanish exactly when the resulting representation is traceless.

\medskip
 
 It is convenient to set $$p(x,y)=T_{s+p}(x)T_{q-r}(y) S_{s}(x)S_{-r}(y)- S_{s+p}(x)S_{q-r}(y)T_{s}(x)T_{-r}(y)$$ and observe that $p(x,y)=p_1(x,y,\tau)S_{s}(x)S_{-r}(y)-p_2(x,y,\tau)S_{s+p}(x)S_{q-r}(y)$.  Thus the projection of $W$ to the $x$-$y$ plane lies in the zero set of the {\em polynomial} $p(x,y)$.  Since $p_1$ and $p_2$ are linear in $\tau$, $\tau$ can be generically eliminated from the system of equations $p_1=0, p_2=0$, so that the projection of $W$ to the zero set of $p(x,y)$ is generically injective.   The precise statement relating $W$ to the zero set of $p$ is awkward to state; see Theorem 11.1  of \cite{HHK}. But it can be summarized by saying that there is a pair of maps with domain $W$, 
 \begin{equation}
\label{2maps}Z=\{(x,y)\in[-1,1]^2~|~ p(x,y)=0\}\leftarrow W \to R(Y_0,T_0)
\end{equation}
such that the right map has generic fibers a single point, and the left map has generic fibers a single point over those $z\in Z$ so that  solving for $\tau$ yields $|\tau|\leq 1$.  In the calculations that follow we first identify $Z$, which is given as the zero set of a single {\em polynomial}, and then identify the   fibers in the  two maps of Equation (\ref{2maps}) to determine $R(Y_0,T_0)$.

 \medskip
 
 \subsection{Restriction to  the pillowcase}
 To find the image of the restriction map $R(Y_0,T_0)$ to the pillowcase $R(S^2,\{a,b,c,d\})$, we use the coordinates $(\gamma,\theta)\in [0,\pi]\times [0,2\pi]$ as indicated in Equation (\ref{eqpil}), corresponding to $a\mapsto\bbi, b\mapsto e^{\gamma\bbk}\bbi$, and $c\mapsto e^{\theta\bbk}\bbi$.  Suppose that $w\in W$.  From Equation (\ref{words}) we see that  the representation sending $A$ to $M=r_w(A)$ and $B$ to $N=r_w(B)$ as in Theorem \ref{thm2} satisfies
 $$\cos(\gamma)=-\Real(ba)= -\Real(N^{-r}M^sM^{s+p}N^{q-r})=-\Real(M^{2s+p}N^{q-2r}).$$
 Similarly
 $$\cos(\theta)=\Real(ca^{-1})=\Real(M^{s+p}N^rM^{-(s+p)}N^{-r})$$
 and
 $$\cos(\theta-\gamma)=\Real(N^{-q}M^{-p}).$$

 These three formulae (together with the angle addition formula) are   sufficient to determine $(\gamma,\theta)\in [0,\pi]\times[0,2\pi]/_\sim$ in terms of $w=(x,y,\tau)$. 
  We omit the straightforward calculation and state the result as follows.

 \begin{prop}\label{sharper} Suppose $(x,y,\tau)\in W$, so that  $$r_w(A)=e^{\arccos(x)\bbi},~ r_w(B)= e^{\arccos(y)e^{\arccos(\tau)\bbk}\bbi}$$ defines a traceless representation.  
 
 Then the image of $r_w$ in the pillowcase has coordinates $(\gamma,\theta)$, where
$$\cos(\gamma)=-T_{2s+p}(x)T_{q-2r}(y) +\sqrt{(1-x^2)(1-y^2)}\, S_{2s+p}(x)S_{q-2r}(y)\, \tau$$
$$\cos(\theta)=
-2 T_{s+p}(x)^2 T_r(y)^2 +2T_r(y)^2+ 2T_{s+p}(x)^2 -1
+2\tau^2\big(1-T_{s+p}(x)^2 -T_{r}(y)^2+ T_{s+p}(x)^2T_r(y)^2\big)
$$
$$\cos(\theta-\gamma)=T_{p}(x)T_{q}(y)-\sqrt{(1-x^2)(1-y^2)}\, S_{p}(x)S_{q}(y)\, \tau.$$ \qed
 \end{prop}
 The expressions for $\cos(\gamma), \cos(\theta), \cos(\theta-\gamma)$ which appear in Proposition \ref{sharper} are rational functions of $x$ and $y$:  the square roots disappear since $p_i(x,y,\tau)=0$.

 \subsection{The (4,5) torus knot}  We begin with a detailed example. In \cite{HHK} the question is asked whether the image of the restriction map to the pillowcase  $R(Y_0,T_0)\to R(S^2,\{a,b,c,d\})$ is always contained in a union of straight lines, as was shown to be true for all 2-bridge knots, all $(2,q)$ torus knots, for the $(4,3)$ torus knot, with respect to natural choices of Conway balls, (as well as for tangles associated to pretzel knots as in Section \ref{pretsec}   below, and, for the traceless binary dihedral representations of {\em any} tangle, by Theorem \ref{bd}).  
 We show that  the image  $R(Y_0,T_0)\to R(S^2,\{a,b,c,d\})$ is non-linear for the tangle of Figure \ref{torusknot} associated to the $(4,5)$ torus knot.

 Remove a Conway ball from $S^3$ as in Figure \ref{torusknot}, by taking 
 $(p,q,r,s)=(4,5,4,-3)$.  Then $s+p=1$ and $q-r=1$, so that
 $$p_1(x,y,\tau)=T_1(x)T_1(y)  -\sqrt{(1-x^2)(1-y^2)}S_1(x)S_1(y) \, \tau =xy-\sqrt{(1-x^2)(1-y^2)}\, \tau,$$
 \begin{eqnarray*}
p_2(x,y,\tau)&=&T_{-3}(x) T_{-4}(y)  -\sqrt{(1-x^2)(1-y^2)}S_{-3}(x) S_{-4}(y) \, \tau \\
&=& 
  \left( 4\,{x}^{3}-3\,x \right)  \left( 1+8\,{y}^{4}-8\,{y}^{2}
 \right) -\sqrt { \left( 1-{x}^{2} \right)  \left( 1-{y}^{2} \right) }
 \left( 1-4\,{x}^{2} \right)  \left( -8\,{y}^{3}+4\,y \right) \, \tau,
\end{eqnarray*}
and
$$p(x,y)=x \left( 16\,{y}^{4}-20\,{y}^{2}+16\,{x}^{2}{y}^{2}-4\,{x}^{2}+3
 \right)
.$$

 Thus   $Z$, the zero set of $p(x,y)$ in the unit square $[-1,1]^2$, is the union of an arc $Z_1=\{(0,y)~|~-1\leq y\leq 1\}$ and the zero set $Z_2$ of $16\,{y}^{4}-20\,{y}^{2}+16\,{x}^{2}{y}^{2}-4\,{x}^{2}+3$. Denote by $W_1$ those triples $(x,y,\tau)\in W$ so that $(x,y)\in Z_1$, and similarly $W_2$ the points lying above $Z_2$.  Clearly
 $$W_1=\{(0,y,0)~|~ y\in[-1,1]\}\cup \{(0,\pm1, \tau)~|~\tau\in[-1,1]\}.$$
 Thus by Theorem \ref{thm2}, the image $W_1\to R(Y_0,T_0)$ is the arc of representations 
 $$y\in [-1,1]\mapsto (M= \bbi, N=e^{\arccos(y)\bbj})$$
This path is  conjugate (by $e^{\frac{\pi}{4}\bbi}$) to a path of binary dihedral representations, and in fact is the arc $R^{tbd}(Y_0,T_0)$ provided by Theorem \ref{bd}. Theorem \ref{thm3.1} gives the image in the pillowcase: it is the straight line segment $(\gamma(t),\theta(t))=(-3t, -8t), t\in[0,\pi]$. This is illustrated by the  blue component in Figure \ref{fig45}.
 
 The component $W_2$ corresponds to the non-binary dihedral representations, and it is more complicated to analyze.
Choose a point $(x,y,\tau)\in W_2$.  Suppose $|x|=1$, then using $p_1=0$ we see that $y=0$, and so $p(x,y)=\pm 1\ne 0$. If $|y|=1$, then $x=0$, but then
$16\,{y}^{4}-20\,{y}^{2}+16\,{x}^{2}{y}^{2}-4\,{x}^{2}+3\ne 0$ contradicting $(x,y,\tau)\in W_2$. Thus $|x|<1$ and $|y|<1$.   Theorem \ref{thm2} then shows that $W_2$ injects into $R(Y_0,T_0)$.

The vanishing of $p_1$  implies that $\tau$ is equal to $xy/\sqrt{(1-x^2)(1-y^2)}$.
Since $|\tau|\leq 1$, it follows that $x^2+y^2\leq 1$.
 
 Solving $16\,{y}^{4}-20\,{y}^{2}+16\,{x}^{2}{y}^{2}-4\,{x}^{2}+3=0$ for $x^2$ one obtains
 $$x^2=\frac{16y^4-20y^2+3}{4-16y^2}=1-y^2 - \frac{1}{4-16y^2}.$$
Then $x^2+y^2\leq1$ implies that $4-16y^2>0$, so that $y^2<\tfrac{1}{4}$. Furthermore $x^2\ge 0$, so that  $16y^4-20y^2+3\ge 0$, which holds precisely for 
$$y\in [-\sqrt{\tfrac{5-\sqrt{13}}{8}},\sqrt{\tfrac{5-\sqrt{13}}{8}}].$$

Hence $W_2$ is homeomorphic to a circle, the union of two arcs given by
\begin{equation}
\label{eq2}
y\in [-\sqrt{\tfrac{5-\sqrt{13}}{8}},\sqrt{\tfrac{5-\sqrt{13}}{8}}],~x=\pm  \sqrt{ 1-y^2 - \tfrac{1}{4-16y^2}},~\tau=\tfrac{xy}{\sqrt{(1-x^2)(1-y^2)}}.
\end{equation}
Also, $W_2$ meets $W_1$ at the two endpoints   of these arcs
$(0,\pm\sqrt{\frac{5-\sqrt{13}}{8}} ,0)$.  Hence $R(Y_0,T_0)$ is homeomorphic to the letter $\phi$.

  To compute the image of $W_2$ in the pillowcase, we use  Proposition \ref{sharper}, which gives:
$$ \cos(\gamma)= -3\,y+4\,{x}^{2}y+4\,{y}^{3}$$
  $$\cos(\theta)=128\,{y}^{8}+128\,{y}^{6}{x}^{2}-256\,{y}^{6}-128\,{y}^{4}{x}^{2}+160
\,{y}^{4}-32\,{y}^{2}+32\,{x}^{2}{y}^{2}+1$$
$$\cos(\theta-\gamma)=-y \left( -16\,{y}^{4}+64\,{y}^{4}{x}^{2}-112\,{x}^{2}{y}^{2}+20\,{y}^
{2}+64\,{x}^{4}{y}^{2}-32\,{x}^{4}+36\,{x}^{2}-5 \right)
$$
Substituting the solution $x^2=\frac{16y^4-20y^2+3}{4-16y^2}$ into these expressions yields
$$\cos(\gamma)={\frac {4{y}^{3}}{-1+4\,{y}^{2}}}$$
$$\cos(\theta)={\frac {-1+12\,{y}^{2}+32\,{y}^{6}-32\,{y}^{4}}{-1+4\,{y}^{2}}}$$
$$\cos(\theta-\gamma)={\frac {4y \left( 16\,{y}^{6}-20\,{y}^{4}+9\,{y}^{2}-1 \right) }{
 \left( -1+4\,{y}^{2} \right) ^{2}}}.
$$
  Finding a simple expression for $\gamma$  and $\theta$   in terms of $y$ is a challenge, but one can numerically compute   at a few sample points $y_i\in  [-\sqrt{\tfrac{5-\sqrt{13}}{8}},\sqrt{\tfrac{5-\sqrt{13}}{8}}]$ to see that $\gamma$ and $\theta$ are not linearly related.  In fact, the curve $(\gamma(y), \theta(y))$ has slope approximately $-1.09$ near $y=\pm\sqrt{\tfrac{5-\sqrt{13}}{8}}$ and is vertical at $y=0$.  Thus, for this tangle, the image of the restriction $R(Y_0,T_0)\to R(S^2,\{a,b,c,d\})$ is {\em not} contained in a union of straight lines.
  
   The space $R(Y_0,T_0)$ and its image in the pillowcase is illustrated in Figure \ref{fig45}. In this figure the blue arc represents the image of the component $W_1$, and the red curve  corresponds to the image of each of the two arcs in $W_2$ corresponding to the two possible choices of square root of $x^2$.  
  \begin{figure}
\begin{center}
\raisebox{.6in}{ \includegraphics[height=1.7in] {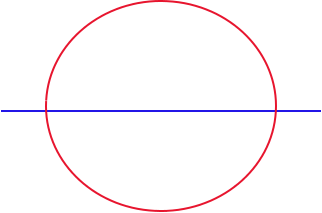}} \includegraphics[height=3in] {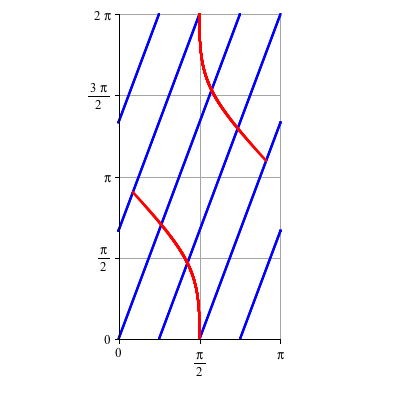}
\caption{$R(Y_0,T_0)$ and its image in the pillowcase for $(p,q,r,s)=(4,5,4,-3)$ \label{fig45} }
\end{center}
\end{figure}
 
The  image of the red curve intersects $R^\nat_\pi(Y_1,T_1)$  (illustrated by the brown curve in Figure \ref{pert}) in the pillowcase    in  two points.  This has four preimages in $R(Y_0,T_0)$: each intersection point gives rise  to a pair on the red circle (symmetric through the vertical axis) corresponding to the two square roots of $x$. The blue curve intersects the brown curve in five points, making a total of $9=|-8|+1$ generators for the instanton complex $IC^\nat(T_{4,5})$, corresponding (see \cite{HHK}) to the fact that the knot signature of $T_{4,5}$ equals $-8$. Five of these generators are binary dihedral, and four are not.    These generators and their images in the pillowcase are illustrated in Figure \ref{fig45a}. In this figure we have drawn the pillowcase as a 2-sphere with four singular points.  The brown curve represents $R^\nat_\pi(Y_1,T_1)$. This is the picture corresponding to Diagram (\ref{SVKD}).

  \begin{figure}
\begin{center}
\includegraphics[height=1.9in] {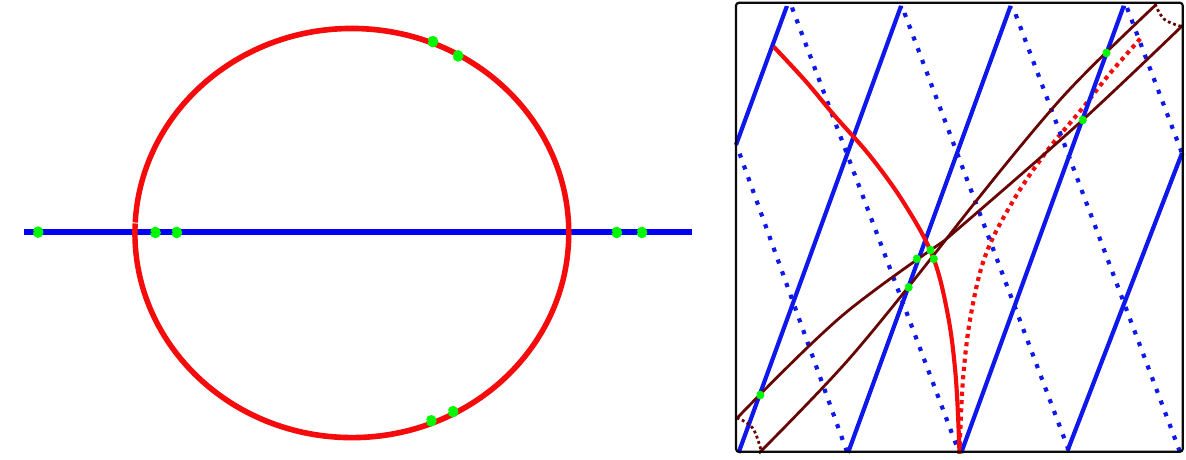}
\caption{$R(Y_0,T_0)$ and its intersection with $R^\nat_\pi(Y_1,T_1)$ in the pillowcase for the $(4,5)$ torus knot. \label{fig45a} }
\end{center}
\end{figure}

It is known (\cite{KM-filtrations}, see also \cite[Section 12]{HHK},  that the $\ZZ/4$ graded instanton homology $I^\nat(T_{4,5})$ has ranks $2,1,2,2$ in gradings $0,1,2,3$ (in brief we write $I^\nat(T_{4,5})=(2,1,2,2)$) so that a non-zero differential must kill a pair of these generators.  
 
\subsection{The (3,7) torus knot} 
In subsequent examples we suppress most details; calculations were carried out using a computer algebra package by precisely the same method described for the (4,5) torus knot. 

For this example  we  take $(p,q,r,s)=(3,7,5,-2)$.    Then 
$$p_1(x,y,\tau)=-x+2\,x{y}^{2}-2\, \sqrt {(1-{y}^{2})(1-{x}^{2})} y\tau,$$
  \begin{multline*}
 p_2(x,y,\tau)= 
-16\,{y}^{5}+20\,{y}^{3}-5\,y+32\,{x}^{2}{y}^{5}-40\,{x}^{2}{y}^{3}+10
\,{x}^{2}y\\ 
 -\sqrt {(1-{y}^{2})(1-{x}^{2})}\big( 2\,      +32\,   {y}^{4}    -24\,    {y}^{2}\big)x\tau,
\hskip2in  \end{multline*}
and
$$p(x,y)=8\,{x}^{2}{y}^{2}-2\,{x}^{2}+32\,{y}^{6}-40\,{y}^{4}+10\,{y}^{2}.$$
The zero set $Z=\{p=0\}$ is a disjoint union of two circles and the arc of Theorem \ref{thm3.1}.    The two circles are exchanged by the  generator of the $\ZZ/2^2$ action on $R(Y_0,T_0)$ which acts trivially on the pillowcase, and hence the two  circles  map  to the same curve (a figure 8) in the pillowcase. This is illustrated in the Figure  \ref{fig37}.

  \begin{figure}
\begin{center}
\raisebox{.6in}{ \includegraphics[height=1.3in] {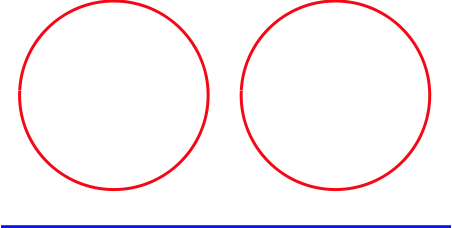}} 
\includegraphics[height=2.7in] {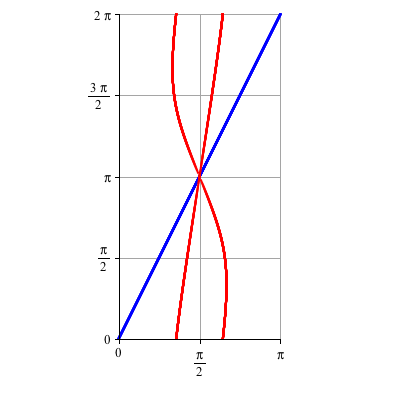}
\caption{$R(Y_0,T_0)$ and its image in the pillowcase for $(p,q,r,s)=(3,7,5,-2)$ \label{fig37} }
\end{center}
\end{figure}
By intersecting with the immersed circle of Figure \ref{pert}, we see that the instanton chain complex is generated by 9 generators: one binary dihedral on the blue arc and eight non-binary dihedral generators on the red circles. For this knot, there are no non-zero differentials in the chain complex and the reduced instanton homology has rank 9 \cite[Section 12.4]{HHK}. Kronheimer-Mrowka construct a spectral sequence converging to the reduced instanton homology of a knot $K$ with $E_2$ term the reduced Khovanov homology of the mirror image of $K$. This shows that the instanton homology for this knot is $(3,2,2,2)$. It seems reasonable to speculate (see \cite[Section 12.6]{HHK}) that the binary dihedral generator contributes in grading 0, and that each red circle contributes four generators, one in each grading.

\subsection{The $(3,10)$ torus knot} We take $(p,q,r,s)=(3,10,7,-2)$.  
Then 
$$ p_1(x,y,\tau)=4\,x{y}^{3}-3\,xy+\sqrt {(1-{y}^{2})(1-{x}^
{2})}\tau(1-4y^2)
$$
and 
 \begin{multline*}p_2(x,y,\tau)=-64\,{y}^{7}+112\,{y}^{5}-56\,{y}^{3}+7\,y+128\,{x}^{2}{y}^{7}-224\,{x
}^{2}{y}^{5}+112\,{x}^{2}{y}^{3}-14\,{x}^{2}y \\  
+   \sqrt {(1-{y}^{2})(1-{x}^
{2})}x\tau    \big(2\, -128\, {
y}^{6}+160\, {y}^{4}-48\,
{y}^{2}\big)
\hskip1in  \end{multline*}

The space $R(Y_0,T_0)$ consists of an arc of binary dihedral representations, a circle meeting this arc in two points, and two more  disjoint circles. The two disjoint circles are immersed into the pillowcase with the same image. The circle meeting the arc is folded and is mapped 2-1 except at the points along the embedded arc of binary dihedral representations. 

  \begin{figure}
\begin{center}
\raisebox{.3in}{ \includegraphics[height=2.3in] {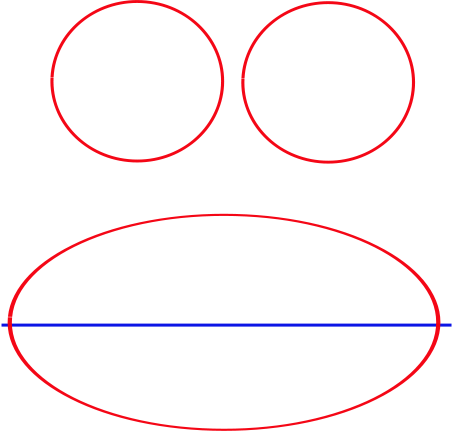}} 
\includegraphics[height=2.7in] {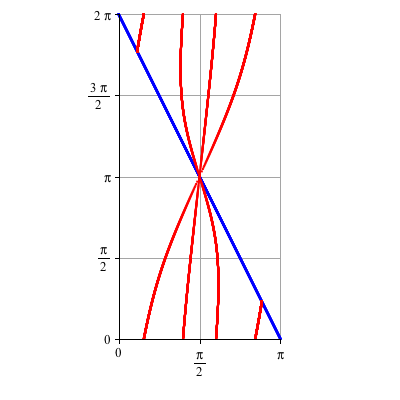}
\caption{$R(Y_0,T_0)$ and its image in the pillowcase for $(p,q,r,s)=(3,10,7,-2)$ \label{fig310} }
\end{center}
\end{figure}
From the figure one sees three generators for the instanton chain complex on the blue binary dihedral arc, four on the first circle, and four on each of the two disjoint circles, for a total of 15 generators. The reduced instanton homology has rank 13, in fact is given by $(4,3,3,3)$.  Thus the chain complex has a non-trivial differential.

\subsection{The $(4,7)$ torus knot}
Figure \ref{fig47} shows $R(Y_0,T_0)$ and its image in the pillowcase for $(p,q,r,s)=(4,7,2,-1)$.  The space $R(Y_0,T_0)$ is a union of two circles and a binary dihedral arc, each circle meeting the arc in two points. The map to the pillowcase folds each circle. 

   \begin{figure}
\begin{center}
\raisebox{.6in}{ \includegraphics[height=1.4in] {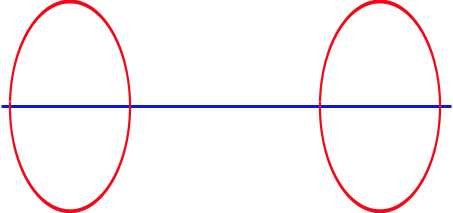}} 
\includegraphics[height=2.4in] {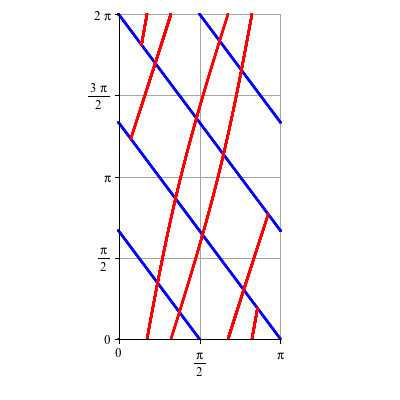}
\caption{$R(Y_0,T_0)$ and its image in the pillowcase for $(p,q,r,s)=(4,7,2,-1)$ \label{fig47} }
\end{center}
\end{figure}
One counts seven generators on the binary dihedral  arc, and four non-binary dihedral generators on each circle, for a total of 15 generators. The instanton homology of the $(4,7)$ torus knot is unknown, but has rank between 11 and 15, thus there may be non-zero differentials in the chain complex.

\subsection{The $(4,9)$ torus knot}   We take $(p,q,r,s)=(4,9,7,-3)$.  

   \begin{figure}
\begin{center}
\raisebox{.6in}{ \includegraphics[height=1.2in] {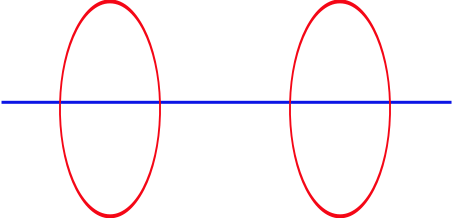}} 
\includegraphics[height=2.6in] {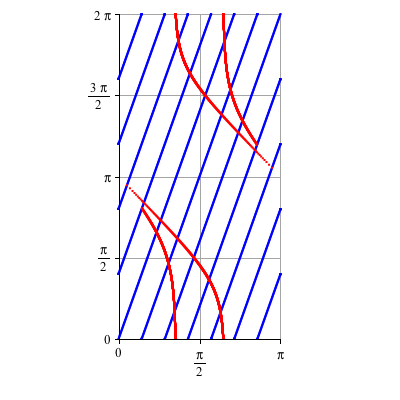}
\caption{$R(Y_0,T_0)$ and its image in the pillowcase for $(p,q,r,s)=(4,9,7,-3)$ \label{fig49} }
\end{center}
\end{figure}

From Figure \ref{fig49} we count 17 generators, four on each circle and nine on the binary dihedral arc.  The instanton complex is $(5,4,4,4)$ up to a mod 4 cyclic permutation (see \cite{HHK}). The instanton homology is unknown, but has rank between 13 and 17.

\subsection{The $(4,11)$ torus knot}  Take $(p,q,r,s)=(4,11,3,-1)$. The space $R(Y_0,T_0)$ is the union of the arc of binary dihedral representations and three circles. The map to the pillowcase folds each circle and embeds the resulting arc.  

   \begin{figure}
\begin{center}
\raisebox{.6in}{ \includegraphics[height=1.2in] {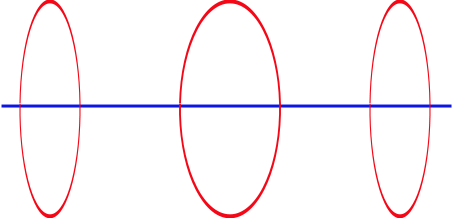}} 
\includegraphics[height=2.8in] {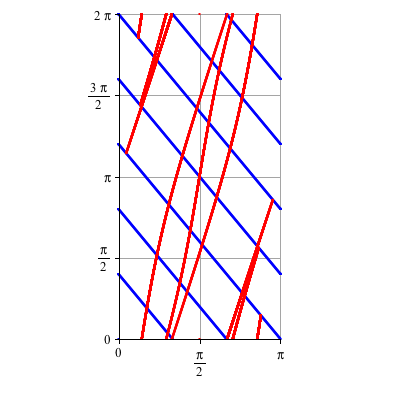}
\caption{$R(Y_0,T_0)$ and its image in the pillowcase for $(p,q,r,s)=(4,11,3,-1)$ \label{fig411} }
\end{center}
\end{figure}

The instanton complex has 11 generators on the arc, and four generators on each circle, giving 23 generators.  The instanton homology is unknown but has rank between 17 and 23.

\subsection{The $(4,13)$ torus knot}  $(p,q,r,s)=(4,13,10,-3)$
The space $R(Y_0,T_0)$ is the union of the arc of binhary dihedral representations and three circles. The map to the pillowcase folds each circle and embeds the resulting arc.   

   \begin{figure}
\begin{center}
\raisebox{.6in}{ \includegraphics[height=1.2in] {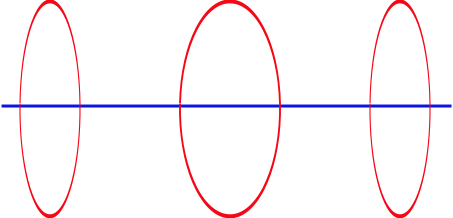}} 
\includegraphics[height=2.6in] {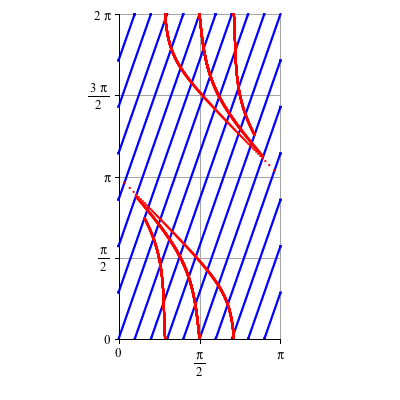}
\caption{$R(Y_0,T_0)$ and its image in the pillowcase for $(p,q,r,s)=(4,13,10,-3)$ \label{fig413} }
\end{center}
\end{figure}

The instanton complex has 13 generators on the arc, and four generators on each circle, giving 25 generators.  The instanton homology is unknown but has rank between 19 and 25.

\subsection{The $(5,7)$ torus knot}
The space $R(Y_0,T_0)$ is illustrated in Figure \ref{fig57} for $(p,q,r,s)=(5,7,3,-2)$.
The circle is immersed by a generically 2-1 immersion. One counts 17 generators for the instanton complex. All differentials are zero in this complex since the reduced instanton homology of the $(5,7)$ torus knot has rank 17.

   \begin{figure}
\begin{center}
\raisebox{.6in}{ \includegraphics[height=1.2in] {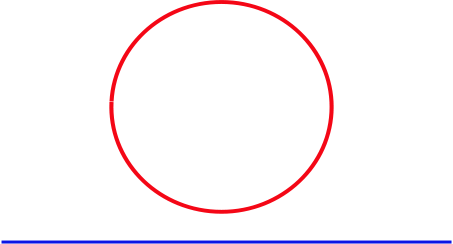}} 
\includegraphics[height=2.6in] {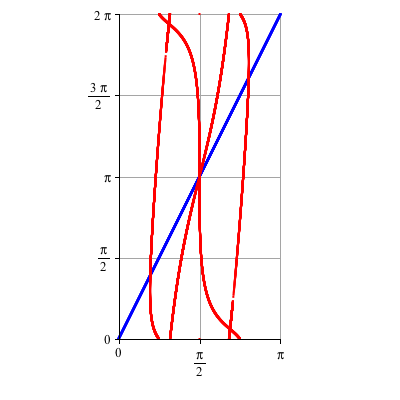}
\caption{$R(Y_0,T_0)$ and its image in the pillowcase for $(p,q,r,s)=(5,7,3,-2)$ \label{fig57} }
\end{center}
\end{figure}

\subsection{The $(5,8)$ torus knot}  We take $(p,q,r,s)=(5,8,5,-3)$.  From Figure \ref{fig58} one counts 21 generators, five on the binary dihedral arc and 16 on the circle. The instanton homology has rank either 19 or 21, so there may be a non-zero differential.

   \begin{figure}
\begin{center}
\raisebox{.6in}{ \includegraphics[height=1.2in] {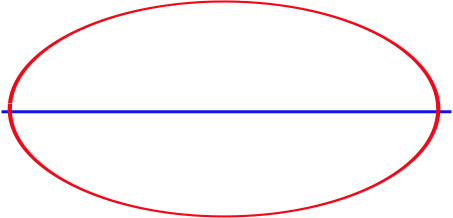}} 
\includegraphics[height=2.6in] {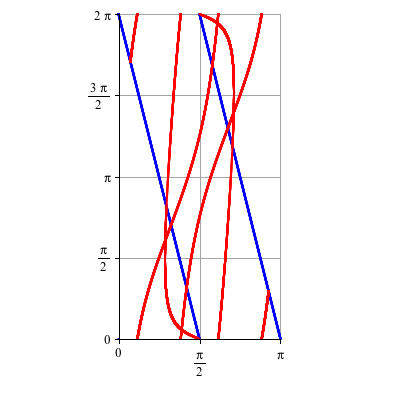}
\caption{$R(Y_0,T_0)$ and its image in the pillowcase for $(p,q,r,s)=(5,8,5,-3)$ \label{fig58} }
\end{center}
\end{figure}

\subsection{The $(5,12)$ torus knot}  We take $(p,q,r,s)=(5,12,5,-2)$. 
The two circles map to the same circle in the pillowcase.  From Figure \ref{fig512} one counts 29 generators, five on the binary dihedral arc and 12 on each circle. All differentials are zero and the reduced instanton homology is $(8,7,7,7)$, up to $\ZZ/4$ cyclic permutation.

   \begin{figure}
\begin{center}
\raisebox{.6in}{ \includegraphics[height=1.2in] {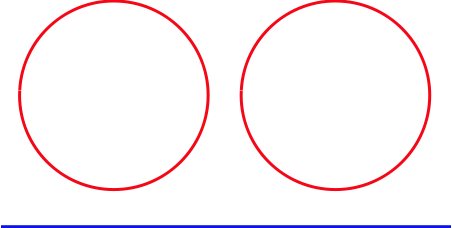}} 
\includegraphics[height=2.6in] {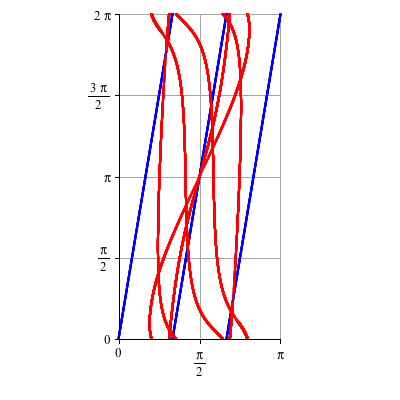}
\caption{$R(Y_0,T_0)$ and its image in the pillowcase for $(p,q,r,s)=(5,12,5,-2)$ \label{fig512} }
\end{center}
\end{figure}

\subsection{The $(6,7)$ torus knot}  We take $(p,q,r,s)=(6,7,6,-5)$.  From Figure \ref{fig67} one counts 19 generators, seven on the binary dihedral arc, four on the inner circle,  and eight on the outer circle. The instanton homology has rank between 11 and 19.

   \begin{figure}
\begin{center}
\raisebox{.6in}{ \includegraphics[height=1.2in] {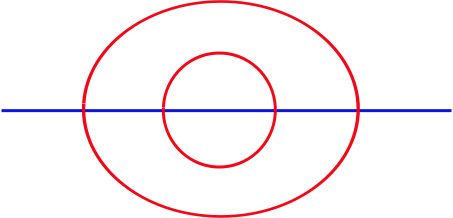}} 
\includegraphics[height=2.6in] {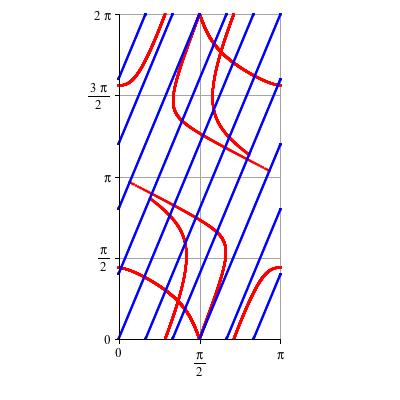}
\caption{$R(Y_0,T_0)$ and its image in the pillowcase for $(p,q,r,s)=(6,7,6,-5)$ \label{fig67} }
\end{center}
\end{figure}

\subsection{The $(5,17)$ torus knot}  Take $(p,q,r,s)=(5,17,7,-2)$. The space $R(Y_0,T_0)$ is a disjoint union of the arc of binary dihedrals and three circles.  There are 41 generators, and all differentials are zero. The instanton homology is $(11,10,10,10)$ up to cyclic permutation.

   \begin{figure}
\begin{center}
\raisebox{.6in}{ \includegraphics[height=1.2in] {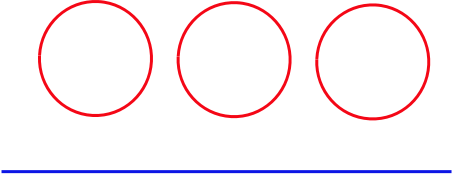}} 
\includegraphics[height=2.6in] {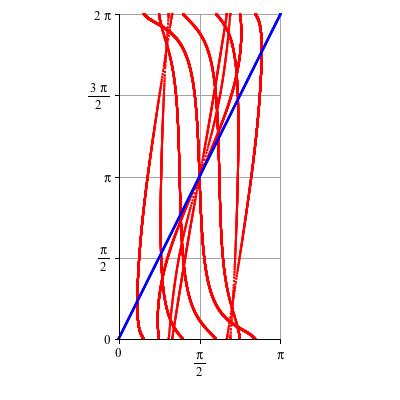}
\caption{$R(Y_0,T_0)$ and its image in the pillowcase for $(p,q,r,s)=(5,17,7,-2)$ \label{fig517} }
\end{center}
\end{figure}

\section{Tangles in Pretzel knot complements}\label{pretsec}

The space $R(S^3,K)$ for $K$ any pretzel knot is identified by Raphael Zentner in \cite{Zen}.  We sketch  what happens when one decomposes a pretzel knot along a Conway sphere, to obtain two tangles. The conclusion is that for an appropriate Conway sphere, the image  of $R(Y_0,T_0)$ in the pillowcase is contained in the union of   straight line segments of two different slopes, one slope corresponding to the binary dihedral representations as explained in Theorem \ref{bd} and the other, coming from  non-binary dihedral representations, due to the fact that the 2-fold branched cover of the tangle is Seifert-fibered. 

We examine three-stranded pretzel knots with two odd twists and one even twist for simplicity, although the argument extends to any pretzel knot. Thus take integers $p,q,r$ with $p$ even and $q,r$ odd and consider the $(p,q,r)$ pretzel knot $K$ and the 3-ball intersecting it in a trivial tangle as illustrated in Figure \ref{pqrpretzel}.  As before, we denote by $(Y_0,T_0)$ the tangle in the complementary 3-ball.
   \begin{figure}
\begin{center}
\def\svgwidth{3.5in}
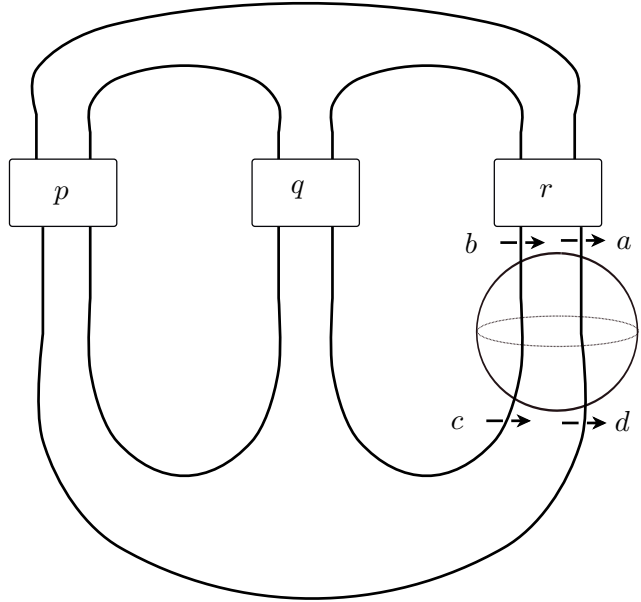
\caption{The $(p,q,r)$ pretzel knot\label{pqrpretzel} }
\end{center}
\end{figure}

Consider the  partition $R(Y_0,T_0)=R^{tbd}(Y_0,T_0)\sqcup R^n(Y_0,T_0)$ of Equation (\ref{deco}).   Theorem \ref{bd} shows that the restriction map of  $R^{tbd}(Y_0,T_0)$ to the pillowcase has image contained in a union of straight line segments of slope $(h(ba)-h(bc^{-1})/h(ba)$, where $h\in H^1(\tilde Y_0)\cong \ZZ$ is a generator.    In contrast to the situation for torus knots explored above, we have the following result.

\begin{thm} \label{naprt} The image of $R^n(Y_0,T_0)$ in the pillowcase is contained in the  union of the two straight line segments 
$$\theta=\gamma(r+1)  \text{~and~} \theta=\gamma(r+1)+\pi.$$
\end{thm}

\noindent{\sl Sketch of proof.}
It is well known \cite{burdez} that the 2-fold branched cover of $S^3$ branched over $K$ is Seifert-fibered over $S^2$. In fact, it is easy to see that if $B_p,B_q,B_r\subset S^3$ are  3-balls containing the twisted bands, then the branched cover of $S^3\setminus (B_p\cup B_q\cup B_r)$ is the product of a 3-punctured sphere and $S^1$ and the branched cover of each 3-ball is a solid torus. Thus the 2-fold branched cover $\tilde Y_0\to Y_0$ branched over $T_0$  is Seifert-fibered over $D^2$ with one (fibered) torus boundary component lying over the Conway sphere $S^2$ illustrated in Figure \ref{pretzel2}.  This is seen by isotoping $B_r$ off the band with $r$ twists.

A (regular) fiber in the boundary torus maps to the loop in $S^2\setminus \{a,b,c,d\}$ expressed as
$$ (ba)^{(r+1)/2}(bc^{-1})(ab)^{(r-1)/2}.$$ 
Hence if $\rho\in R(Y_0,T_0)$ sends $a$ to $\bbi$, $b$ to $e^{\gamma \bbk}\bbi$, and $c$ to $e^{\theta\bbk}\bbi$, the regular fiber is sent to $(-1)^re^{(\gamma(r+1)-\theta)\bbk}$.

A traceless representation $\rho:\pi_1(Y_0\setminus T_0)\to SU(2)$ restricts on the index 2 subgroup to a representation $\rho:\pi_1(\tilde Y_0\setminus \tilde T_0)\to SU(2)$ which sends the meridians of $\tilde T_0$ to $-1$.  It follows that composing with the adjoint homomorphism $SU(2)\to SO(3)$ gives a representation $ad\rho$ which send the meridians to 1, and hence extends to the branched cover $\tilde Y_0$
$$ad\rho:\pi_1(\tilde Y_0)\to SO(3). $$
We partition $R(Y_0,T_0)$ into two subsets 
$$R(Y_0,T_0)=R^a(Y_0,T_0)\sqcup R^{na}(Y_0,T_0)$$
corresponding to whether $ad\rho$ is abelian or non-abelian.

If $ad\rho$ is non-abelian, then the regular fiber, being a central element of $\pi_1(\tilde Y_0)$, must be sent to $1\in SO(3)$, and hence to $\pm 1$ by $\rho$. Thus if $\rho\in R^n(Y_0,T_0)$,   $$\pm1=(-1)^re^{(\gamma(r+1)-\theta)\bbk} $$
so that $\gamma(r+1)-\theta\equiv 0\mod \pi$.  Notice, moreover, that since $ad\rho$ is non-abelian, $\rho$ restricts to a non-abelian representation on the index 2 subgroup $\pi_1(\tilde Y_0\setminus \tilde T_0)$ and hence is not binary dihedral.  Thus $R^{na}(Y_0,T_0)\subset R^n(Y_0,T_0)$.

If, on the other hand, $ad\rho$ is abelian, then since $\tilde Y_0$ is a $\ZZ/2$ homology sphere, the image of $ad\rho$ may be conjugated to lie in a maximal torus, and hence $\rho:\pi_1(\tilde Y_0\setminus \tilde T_0)\to SU(2)$ takes values in a maximal torus in $SU(2)$.
Conjugate $\rho$ if necessary so that $\rho(a)=\bbi,~\rho(b)=e^{\gamma\bbk}\bbi$, and $\rho(c)=e^{\theta\bbk}\bbi$. Choose an imaginary quaternion ${\bf q}$ so that
the restriction $\rho:\pi_1(\tilde Y_0\setminus \tilde T_0)\to SU(2)$ takes its values in the unique maximal torus containing ${\bf q}$.  

Suppose that for some $\alpha \in\pi_1(\tilde Y_0\setminus \tilde T_0)$,  $\rho(\alpha)=e^{x{\bf q}}$ for some $x\not\equiv 0\mod \pi$.  Then $a\alpha a^{-1}$ lies in $\alpha \in\pi_1(\tilde Y_0\setminus \tilde T_0)$, so that $\rho(a\alpha a^{-1})=e^{-x{\bf \bbi q\bbi}}$.  Hence ${\bf \bbi q\bbi}=\pm {\bf q}$.  Similarly 
$e^{\gamma\bbk}\bbi {\bf q} e^{\gamma\bbk}\bbi=\pm  {\bf q}$ and  $e^{\theta\bbk}\bbi {\bf q} e^{\theta\bbk}\bbi=\pm  {\bf q}$.  This implies that either $\gamma$ and $\theta$ are multiples of $ \pi$ and ${\bf q}=\pm \bbi$, or else ${\bf q}=\pm \bbk$.  

In the first case, given any other meridian $\mu$, $\rho(\mu)\bbi=\rho(\mu a)=e^{x\bbi}$ for some $x$. Hence $\rho(\mu)=e^{(x-\frac{\pi}2)\bbi}$, which must equal $\pm\bbi$ since $\rho$ is traceless. Thus $\rho$ takes values in $\{\pm 1,\pm\bbi\}$, and hence is conjugate (by $e^{-\frac{\pi}4 \bbj}$) to a traceless binary dihedral representation. Its image in the pillowcase is one of the corner points. 

In the second case, any other meridian $\mu$ must satisfy $\rho(\mu)=e^{x\bbk}\bbi$, since $\rho(\mu)\bbi=\rho(\mu a)\in \{e^{x\bbk}\}$.  Thus $\rho$ is traceless binary dihedral.

Hence $R^a(Y_0,T_0)\subset R^{tbd}(Y_0,T_0)$. Since $R^{na}(Y_0,T_0)\subset R^n(Y_0,T_0)$ and 
$$R^{tbd}(Y_0,T_0)\sqcup R^{n}(Y_0,T_0)=R(Y_0,T_0)=R^a(Y_0,T_0)\sqcup R^{na}(Y_0,T_0),$$ it follows that 
$$R^{tbd}(Y_0,T_0)=R^a(Y_0,T_0)\text{~and~} R^{n}(Y_0,T_0)=R^{na}(Y_0,T_0).$$
In particular, the restriction $R^{n}(Y_0,T_0)\to R(S^2,\{a,b,c,d\})$ takes values in the two line segments  given by $\gamma(r+1)-\theta\equiv 0\mod \pi$.

\qed

\bigskip

The space $R(Y_0,T_0)$  for the $(-2,3,r)$ pretzel knot  for any odd positive integer $r$ is illustrated in Figure \ref{pretzel2}.  Theorem \ref{bd} shows that the subspace $R^{tbd}(Y_0,T_0)$ is an arc parametrized by $[0,\pi]$. With a bit more work one can show that the closure  of  $R^{n}(Y_0,T_0)$ is a circle meeting   $R^{tbd}(Y_0,T_0)$ in two points corresponding to the parameters $\frac{\pi}{6}$ and $\tfrac{5\pi}{6}$. The two open semicircles making up $R^{n}(Y_0,T_0)$ are exchanged by the action of the appropriate character $\chi$   
described in Proposition \ref{symmetry}.

 Theorem \ref{bd} implies that $R^{tbd}(Y_0,T_0)$ maps to the straight line segments $\theta=(r-5)\gamma,~\gamma\in [0,\pi]$.   Theorem    \ref{naprt} implies that the image of $R^{n}(Y_0,T_0)$ in the pillowcase  is  the segment  $\theta=(r+1)\gamma+\pi,~\gamma\in [\tfrac{\pi}{6},\tfrac{5\pi}{6}]$, Each of the two semi-circles of $R^{n}(Y_0,T_0)$ is mapped injectively to the interior of this arc.

   \begin{figure}
 \begin{center}{ \includegraphics[height=1.7in] {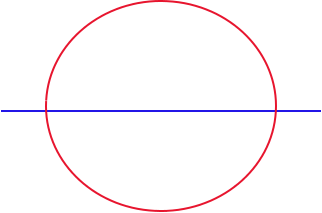}} 
 \caption{$R(Y_0,T_0)$  for the $(-2,3,r)$ pretzel knot \label{pretzel2} }
 \end{center}
\end{figure}

We illustrate in Figure \ref{237}  the explicit case for the $(-2,3,7)$ pretzel knot. From this picture one counts 9 generators for the reduced instanton complex: one on the arc of binary dihedral representations and four each on the two semi-circles $R^n(Y_0,T_0)$. The Alexander polynomial for this knot is $t^5-t^4 + t^2-t + 1-t^{-1} + t^{-2}-t^{-4} + t^{-5}$, from which Kronheimer-Mrowka's theorem implies that all differentials in the corresponding chain complex vanish.

The reduced Khovanov homology of the mirror image of this knot is computed by the  Knot theory' program \cite{BN}  to equal (3,2,2,2) mod 4. Thus there are no higher differentials in the Kronheimer-Mrowka spectral sequence, and this calculation also gives the reduced instanton homology.

    \begin{figure}
\begin{center}{ \includegraphics[height=2.7in] {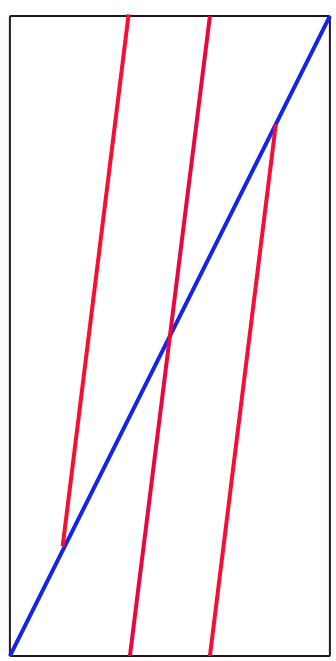}} 
\caption{The image of  $R(Y_0,T_0)$  in the pillowcase for the $(-2,3,7)$ pretzel knot \label{237} }
\end{center}
\end{figure}

For general $(-2,3,n)$ pretzel knots, intersecting $R(Y_0,T_0)$ with $R^\nat_\pi(Y_1,T_1)$  (illustrated in Figure \ref{pert}) in the pillowcase, one counts
$1+2\frac{n-7}{2}=n-6$ binary dihedral generators, and $4([\frac{5n+6}{12}]-[\frac{n+6}{12}]-1)$ non-binary dihedral generators for the reduced instanton chain complex (where $[q]$ denotes the greatest integer less than $q$).  Therefore
$$\text{Rank}~ IC^\nat(S^3,K)=n-2 +4\big([\tfrac{5n+6}{12}]-[\tfrac{n+6}{12}]). $$
 
\section{Speculation and generalization} 

The  data  produced for torus knots is consistent with the following   conjecture:

\medskip

\noindent{\bf Conjecture.}  {\sl Suppose that $K\subset S^3$ is a knot decomposed  as in Equation (\ref{decom}) with $(Y_1,K_1)$ a trivial tangle and $R(Y_0,T_0)$ non-degenerate (in the sense of \cite[Section 7]{HHK}), so that no holonomy perturbations in $Y_0$ are needed to construct the reduced instanton complex.  Suppose further that the closure of $R^n(Y_0,T_0)$ misses $R^{tbd}(Y_0,T_0)$. Then all differentials are zero in the reduced instanton complex.
}

\medskip

In an different direction, we point out that the entire picture painted in this article has a complex counterpart obtained by considering $SL(2,\CC)$ representations of 2-stranded tangles and the associated character varieties.     Indeed, set
$$R_\CC(Y_0,T_0)=\{\rho:\pi_1(Y_0\setminus T_0)\to SL(2,\CC)~|~{\text tr }\rho(\mu)=0 \text{ for each meridian } \mu\}/\hskip-.7ex / _{SL(2,\CC)},$$
where, in the context of $SL(2,\CC)$, we identify representations $\rho_1,\rho_2$ if the characters $\tr(\rho_1(\gamma))$ and $\tr(\rho_2(\gamma))$ are equal for all $\gamma\in \pi_1(Y_0\setminus T_0)$.  This is  a slightly weaker relation than conjugacy and corresponds to the GIT quotient corresponding to the conjugation action.

One can show  that $
R_\CC(S^2,\{a,b,c,d\})$
is the {\em complex pillowcase} $(\CC^\times\times \CC^\times) /\ZZ/2$, where $\ZZ/2$ acts on $\CC^\times\times \CC^\times$ by $(\lambda,\mu)\mapsto(\lambda^{-1},\mu^{-1})$.  Explicitly, the map $$ \CC^\times\times \CC^\times\to R_\CC(S^2,\{a,b,c,d\})$$ given by 
$$(\mu,\lambda)\mapsto \big( a\mapsto \begin{pmatrix} 0&1\\-1&0\end{pmatrix},~b\mapsto \begin{pmatrix} 0&\lambda\\-\lambda^{-1}&0\end{pmatrix},~c\mapsto \begin{pmatrix} 0&\mu\\-\mu^{-1}&0\end{pmatrix}\big)$$
induces a homeomorphism of $(\CC^\times\times \CC^\times) /\ZZ/2$ with 
$R_\CC(S^2,\{a,b,c,d\})$. The $SU(2)$ pillowcase embeds as the subspace $(S^1\times S^1)/\ZZ/2$.  
The image of the restriction map 
$$R_\CC(Y_0,T_0)\to R_\CC(S^2,\{a,b,c,d\})$$
gives an invariant of tangles, extending the $SU(2)$ case focused on in this article.

A variant is to consider the map 
$$R_\CC(S^2,\{a,b,c,d\})\to \CC^2,~ \rho\mapsto (\text{tr}(\rho(ba^{-1}), \text{tr}(\rho(ca^{-1})).$$ The composite 
$$R_\CC(Y_0,T_0)\to R_\CC(S^2,\{a,b,c,d\})\to \CC^2$$
is a plane curve whose defining polynomial 
gives a tangle analogue of the  $A$-polynomial of \cite{ccgls}.

\end{document}